\documentclass[10pt]{amsproc}
\usepackage{amsmath}
\usepackage{amssymb}
\usepackage{amsfonts}
\usepackage{amsthm}
\usepackage{amsxtra}
\usepackage{mathtools}
\usepackage[toc,page]{appendix}
\usepackage{rotating}
\usepackage{pdflscape}
\usepackage{nicefrac}
\usepackage{float}
\usepackage{array}
\usepackage{pifont}
\usepackage{multirow}
\usepackage{caption}
\usepackage{subcaption}
\usepackage{graphics}
\usepackage{enumitem}
\usepackage[super]{nth}
\usepackage{datetime}
\usepackage{hyperref}
\hypersetup{
  colorlinks   = true, 
  urlcolor     = green, 
  linkcolor    = blue, 
  citecolor   = red 
}
\usepackage{xcolor}

\theoremstyle{plain}

\newtheorem{corollary}{Corollary}[section]
\newtheorem{lemma}{Lemma}[section]
\newtheorem{proposition}{Proposition}[section]
\theoremstyle{definition}
\newtheorem{definition}{Definition}[section]

\theoremstyle{remark}
\newtheorem{remark}{Remark}[section]

\newdateformat{monthdayyeardate}{%
  \monthname[\THEMONTH]~\THEDAY, \THEYEAR}

\newcommand{\C}{\mathbb C}

\newcommand{\Z}{\mathbb Z}
\newcommand{\N}{\mathbb N}

\newcommand{\half}{
        {\lower0.00ex\hbox{\raise.6ex\hbox{\the\scriptfont0 1}
                           \kern-.5em\slash\kern-.1em\lower.45ex
                                     \hbox{\the\scriptfont0 2}}}}
\newcommand{\quarter}{
        {\lower0.00ex\hbox{\raise.6ex\hbox{\the\scriptfont0 1}
                           \kern-.5em\slash\kern-.1em\lower.45ex
                                     \hbox{\the\scriptfont0 4}}}}
\newcommand{\tquarter}{
        {\lower0.00ex\hbox{\raise.6ex\hbox{\the\scriptfont0 3}
                           \kern-.5em\slash\kern-.1em\lower.45ex
                                     \hbox{\the\scriptfont0 4}}}}
\newcommand{\eighth}{
        {\lower0.00ex\hbox{\raise.6ex\hbox{\the\scriptfont0 1}
                           \kern-.5em\slash\kern-.1em\lower.45ex
                                     \hbox{\the\scriptfont0 8}}}}
\newcommand{\othird}{
        {\lower0.00ex\hbox{\raise.6ex\hbox{\the\scriptfont0 1}
                           \kern-.5em\slash\kern-.1em\lower.45ex
                                     \hbox{\the\scriptfont0 3}}}}

\begin{document}

\title[]{An Infinite Product of the Incomplete Beta Function-type Hypergeometric Function and its Probabilistic Origins}

\author{N. S.~Witte}
\address{School of Mathematics and Statistics, Victoria University of Wellington, PO Box 600 Wellington 6140, New Zealand}
\email{\tt n.s.witte@protonmail.com}

\begin{abstract}
Recently it has been shown that the $\alpha$-Sun density $h(x)$ [{\it J. Math. Anal. Appl.}, {\bf 527} (2023), p. 127371]
which interpolates between the Fr{\'e}chet density and that of the positive, stable distributions whose density is given by a Fox $H$-function,
has a Mellin transform involving an infinite product of ratios of Incomplete Beta functions.
We develop systematic, but asymptotic, approximations for such products and consequently for the behaviour of the density as $ x\to 0+$ which
complement the recent exact form for this by Simon [{\it Electron. Commun. Probab.}, {\bf 28} (2023) p. 1 - 13].
The systematic expansion is an example of a Power Product Expansion, 
and in our case we derive bounds and estimates which show that this expansion is not convergent and thus only yields an asymptotic expansion.
\end{abstract}

\maketitle

\section{The $\alpha$-Sun density}\label{Introduction}
\setcounter{equation}{0}
In \cite{WG_2023} the analytic properties of a density function 
$ h(x;\alpha,\gamma) $, $ x \in (0,\infty) $, $ \gamma > 0 $, $ 0 < \alpha < 1 $
which arises from the domain of attraction problem for a statistic interpolating between the supremum and sum of random variables were investigated.
The parameter $ \alpha $ controls the interpolation between these two cases, 
while $ \gamma $ parametrises the type of extreme value distribution from which the underlying random variables are drawn.
For $ \alpha = 0 $, $ \gamma>0 $ the density reduces to the Fr{\'e}chet density, 
whereas for $ \alpha = 1 $, $ 0< \gamma < 1 $ it is a particular Fox $H$-function appropriate to positive, stable distributions with index $\gamma$.
However for intermediate $ \alpha $ an entirely new function appears, 
which is not one of the extensions to the hypergeometric function considered to date.

In their study of this model - the $\alpha$-Sun process - Greenwood and Hooghiemstra \cite[Theorem 2, Eq. (2.4)]{GH_1991}
have shown this density $h(x;\alpha,\gamma)$ satisfies a linear, homogeneous integral equation 
\begin{equation}
  h(x) = \frac{\gamma}{x} \int_{0}^{x}du\; \frac{h(u)}{(x-\alpha u)^{\gamma}} ,
\label{integral_eqn}
\end{equation}
subject to the conditions: $ \alpha \in (0,1) $, $ \gamma \in (0,\infty) $ identifies the domain of attraction as above, 
and $ h(x) $ is a real, normalised probability density on $ x \in (0,\infty) $.
It was observed that the $\alpha$-Sun process, 
for all $\alpha$ in $(0,1)$, behaves in a way that is much more like the supremum of the input variables than like the sum. 
Furthermore there is non-trivial limiting behaviour as $\alpha \to 1^{-}$.

The central result of \cite{WG_2023} gives a Mellin-Barnes integral representation for $h(x)$, Eq.(3.41) in Prop. 3.9, 
which is a standard way of defining many special functions, such as the Meijer-$G$ and Fox-$H$ functions.
Under the conditions $ 0 \leq \alpha < 1 $, $ \gamma>0 $, $ 0<x<\infty $ the density $ h(x;\alpha,\gamma) $ has the representation for any $ c < 1 $,
\begin{equation}
	h(x) = \frac{\gamma}{2\pi ix}\int_{c-i\infty}^{c+i\infty} dt\; x^{-\gamma t}\Gamma(1-t)(1-\alpha)^{-\gamma t}
	\prod_{j=1}^{\infty}\frac{{}_2F_1(\gamma,(j-t)\gamma;1+(j-t)\gamma;\alpha)}{{}_2F_1(\gamma,j\gamma;1+j\gamma;\alpha)} ,
\label{MBdensity}
\end{equation}
where $ {}_2F_1(a,b;c;z) $ is the Gauss hypergeometric function \cite[\S 15]{DLMF}. 
In the general case the Mellin transform of the density $ h(x) $ is defined by
\begin{equation}
    H(s) := \int_{0}^{\infty} dx\; x^{s-1}h(x) ,
\label{MT_defn}
\end{equation}
for the vertical strip in the complex $s$-plane, $1-\gamma < {\rm Re}(s) < 1+\gamma$.
Prop. 3.1 of \cite{WG_2023} states that for $ \Re(s) < 1+\gamma $ the Mellin transform $ H(s;\alpha,\gamma) $ satisfies the linear, homogeneous functional equation
\begin{equation}
	H(s) = \frac{\gamma}{1+\gamma-s} {}_2F_1(\gamma,1+\gamma-s;2+\gamma-s;\alpha) H(s-\gamma) .
\label{MT_functional}
\end{equation}

Firstly with regard to the analytical character of ${}_2F_1(\gamma,1+\gamma-s;2+\gamma-s;\alpha)$ with respect to $ s $ we can make the following immediate observations:
For $ 0 < |\alpha| < 1 $ the ${}_2F_1$ in \eqref{MT_functional} is finite for all $ \Re(s) < 1+\gamma $ 
whereas it has simple poles on the right-hand side of the complex $s$-plane at $ s=1+\gamma+k $ for $ k \in \Z_{\geq 0} $.
For $ |\alpha| = 1 $ we require in addition that $ \Re(\gamma) < 1 $ and the poles are the same as for $ |\alpha| < 1 $.
This ${}_2F_1$ function is one of a pair of solutions to the hypergeometric differential equation about $ \alpha=0 $ that is bounded as $ s \to -\infty $ (by $ (1-\alpha)^{-\gamma} $). 
This distinguishes itself from the other member of the pair which has the simple algebraic form $ \alpha^{s-\gamma-1} $, and which diverges as $ s \to -\infty $.
This hypergeometric function can be linearly mapped into other forms via the Kummer or linear fractional transformations of the $ \alpha $-plane
and we record one version \cite[15.8.1]{DLMF} of these:
\begin{equation*}
   {}_2F_1(\gamma,1+\gamma-s;2+\gamma-s;\alpha) 
   = (1-\alpha)^{-\gamma} {}_2F_1(\gamma,1;2+\gamma-s;\tfrac{\alpha}{\alpha-1}) .
\end{equation*}
A useful alternative integral representation of the $ {}_2F_1 $ is
\begin{equation}
   {}_2F_1(\gamma,1;2+\gamma-s;\tfrac{\alpha}{\alpha-1}) = 1-\alpha\gamma(1-\alpha)^{\gamma} \int_{0}^{1} dt\; t^{1+\gamma-s}(1-\alpha t)^{-\gamma-1} .
\label{F_integral}
\end{equation}
The hypergeometric function can be identified in a number of ways.
In one instance it is an incomplete Beta function, see \cite[8.17.7]{DLMF},
\begin{equation*}
   {}_2F_1(\gamma,1+\gamma-s;2+\gamma-s;\alpha) = (1+\gamma-s)\alpha^{s-\gamma-1}B_{\alpha}(1+\gamma-s,1-\gamma) . 
\end{equation*}
In particular a sequence of these $ {}_2F_1 $ will feature prominently.
\begin{definition}\label{F-defn}
Let $ F_{j} := {}_2F_1(\gamma,j\gamma;1+j\gamma;\alpha) $ for $ j \in \N_{0} $. 
The series expansion in $\alpha$ also generates a partial fraction form with respect to $j$ which will be useful in the sequel
\begin{equation}
	F_{j} = \sum_{l=0}^{\infty} \alpha^{l}\frac{(\gamma)_{l}}{l!}\frac{j}{j+l\gamma^{-1}} .
\label{F-ParFrac}
\end{equation}
\end{definition}
Some basic properties of this sequence are the following. Let $ 0 < \alpha < 1 $ and $ \gamma > 0 $.
The hypergeometric factors $ F_j $ satisfy $ F_{0} = 1 $, are monotonically increasing, $ F_{j+1} > F_{j} $, and are bounded above, $ F_{j} < (1-\alpha)^{-\gamma} $.

Simon has recently given the asymptotic form of the density, \cite[Thm. 2]{Sim_2023}, as $ x\to 0^{+} $,
\begin{equation}
	h(x;\alpha,\gamma) \mathop{\sim}\limits_{x \to 0^+} c(\alpha,\gamma) x^{-1-\gamma(1-\alpha)^{-1}} e^{-(1-\alpha)^{-\gamma}x^{-\gamma}} ,
\label{h_x2zero}
\end{equation}
where the constant factor is
\begin{equation}
	c = \gamma(1-\alpha )^{\frac{\gamma}{\alpha-1}}\exp\left(\frac{\alpha\psi(1)}{\alpha-1}\right)
		\prod^{\infty}_{k=1}\frac{\exp \left(\frac{\alpha}{(\alpha-1)k}\right)}{{}_2F_1\left(1,\gamma ;1+k\gamma;\frac{\alpha}{\alpha-1}\right)} .
\label{constant}
\end{equation}

The modus operandi of our study is to develop alternative representations for the infinite products of hypergeometric functions, 
ideally as identities or alternatively as approximations in any sense, 
especially those which transparently exhibit the behaviour of the density as the parameters $ \alpha \to 1^{-} $ and for all $\gamma$.
These are the regions where the magnitudes of the products grow or shrink the fastest, and have an interesting interpretation for the underlying probability models.
The infinite products are very slowly converging and are completely opaque regarding even the gross behaviour of the product in these parameter regions.
Simon's constant in the regime $ x\to 0^{+} $ is no more helpful in this regard, although it is a significant step to know this result.
In keeping with this aim we develop a product formula for the $F_j$ factors themselves using a recursive formula in \S \ref{ProductAnsatz}.
In the following section \S \ref{PPE} we situate our product formula within a class of generating functions known as power product expansions and generalised Lambert series,
recalling a number of known results for recursive and convolution identities for relations between these objects in a general setting.
Using these relations we derive bounds and estimates in \S \ref{Bounds+Growth} for our particular application, 
and demonstrate the asymptotic character of our product formula due to the super-exponential growth of their coefficients.
Our final section \S \ref{ProductsofGamma} employs our product formula for the $ F_j $ in the product over $j$, 
switches the order and employs an identity for products of rational roots-of-unity factors in terms of Gamma functions to arrive at an asymptotic approximation for the density,
given in Prop. \ref{density_approx}. 
Some simple Gamma function product formulae are given for the first two successive approximations to the constant factor in the \eqref{1stApproxDensity} and \eqref{2ndApproxDensity},
and these are numerically compared to Simon's exact formula.

\section{The Product Formula}\label{ProductAnsatz}
\setcounter{equation}{0}

From Prop. 3.5 of \cite{WG_2023} the consecutive product of hypergeometric functions $F_{k}$ has the leading order behaviour as $ N \to \infty $,
\begin{equation}
   \prod_{k=1}^{N} {}_2F_1(\gamma,k\gamma;1+k\gamma;\alpha) \mathop{\sim}\limits_{N \to \infty} 
   C (1-\alpha)^{-N\gamma}\left[ N+1+\frac{1}{\gamma} \right]^{-\alpha/(1-\alpha)} ,
\label{2F1-product_asymptotic:b}
\end{equation}
where $ C $ is a constant independent of $ N $.
Another useful function turns out to be
\begin{equation}
   G(t) := \frac{1}{\Gamma\left(\dfrac{1+\gamma-s}{\gamma}\right)}H(s) ,\quad t := \frac{1-s}{\gamma} ,
\label{G_defn}
\end{equation}
which eliminates the simple poles at $ s = 1+k\gamma $, $ k \in \N $, so that in the finite-$t$ plane
$ |G(t)| \leq (1-\alpha)^{-\gamma} \left[ \cosh(\pi\Im(t)) \right]^{1/2} $.
At the prescribed values $ s=1-k\gamma $, $ k \in \Z_{\geq 0} $, we deduce 
\begin{equation}
   G_{k} := G(t=k) = \frac{1}{\prod_{l=1}^{k} {}_2F_1(\gamma,l\gamma;1+l\gamma;\alpha)} ,
\label{G-values}
\end{equation}
which have the growth properties given above.

Furthermore the finite product $ G_k $, defined by \eqref{G-values}, 
was shown to be \cite[Lemma 3.3]{WG_2023} a function of complex $k$ analytic in the half-plane $ \Re(k) > 1+\gamma^{-1} $ in an infinite product form,
\begin{equation}
   G_{k} = (1-\alpha)^{\gamma k} \prod_{j=1}^{\infty} \frac{F_{j+k}}{F_{j}} ,
\label{Gk}
\end{equation}
and has an interpolating function 
\begin{equation}
   G(-t) = (1-\alpha)^{-\gamma t} \prod_{j=1}^{\infty} \frac{F_{j-t}}{F_{j}} .
\label{Ginterpolate}
\end{equation}

Our new result is really a large-$j$ expansion for $ F_{j}$ but not in the usual asymptotic sense of Poincar\'e with a sum of descending monomials 
but rather in the sense of a Pad\'e approximant with only denominator factors.
\begin{proposition}\label{ProductFormula}
Let $ 0\leq \alpha < 1 $, $ \gamma > 0 $.
There is a sequence $\{ f_{m} (\alpha, \gamma) \}_{m=1}^{\infty} $ such that for each $m \in \N$,
\begin{multline}
	(1-\alpha)^{\gamma} \left( 1+\frac{1}{j}f_{1} \right)\left( 1+\frac{1}{j^2}f_{2} \right) \cdots
	\left( 1+\frac{1}{j^m}f_{m} \right)F_{j} 
\\
	= 1 + {\rm O}(j^{-m-1}), \quad \text{\rm as $j\to \infty$} .
\label{large-j_Fj}
\end{multline}
\end{proposition}
\begin{proof}
The values of $f_m$ for $m = 1, 2, 3, . . .$ can be generated by hand or otherwise by expanding successive products with $ F_{j} $ as a series expansion in large $j$,
see \eqref{F-ParFrac}.
For example to begin with one observes that the leading term $ {\rm O}(j^{-1}) $ is
\begin{equation*}
	\log\left[ (1-\alpha )^{\gamma }{}_2F_1(\gamma ,j \gamma ;j \gamma +1;\alpha ) \right] 
	= -\frac{\alpha}{1-\alpha}j^{-1} + {\rm O}(j^{-2}) ,
\end{equation*} 
and that to eliminate this requires multiplication by an additional factor $ (1+\frac{\alpha}{1-\alpha}j^{-1} ) $.
Re-expanding again one observes a new leading term of order $ {\rm O}(j^{-2}) $,
\begin{equation*}
	\log\left[ (1-\alpha )^{\gamma } \left( 1+\frac{\alpha }{(1-\alpha )}j^{-1} \right) \, {}_2F_1(\gamma ,j \gamma ;j \gamma +1;\alpha ) \right]
	= \frac{\alpha }{\gamma(1-\alpha )^2}j^{-2} + {\rm O}(j^{-3}) .
\end{equation*}
And so on. We will now show that this holds in general for all $m>1$ by induction. 
Having assumed its veracity for $m$, i.e., 
\begin{multline*}
	(1-\alpha)^{\gamma} \left( 1+\frac{1}{j}f_{1} \right)\left( 1+\frac{1}{j^2}f_{2} \right) \cdots	\left( 1+\frac{1}{j^m}f_{m} \right) F_{j}
\\
	= 1 + {\rm O}(j^{-m-1}) = 1 + C_{m+1}j^{-m-1} + {\rm O}(j^{-m-2}),
\end{multline*}
where the sub-leading terms indicated above contain decreasing powers of $j$ which drop by one unit each term.
It should be observed that $C_{m+1}$ has no dependence on $j$ 
and by algebraic construction from terms arising from the expansion of the products on the left-hand side it can only have contributions where the sum of non-zero indices add up to $m+1$,
that is to say it can only have dependence on $f_{1},\ldots,f_{m}$.
This follows from the fact that $F_{j}$, by \eqref{F-ParFrac}, has an analytic expansion about $j=\infty$.
Now we compute
\begin{multline*}
	(1-\alpha)^{\gamma} \left( 1+\frac{1}{j}f_{1} \right)\left( 1+\frac{1}{j^2}f_{2} \right) \cdots	\left( 1+\frac{1}{j^m}f_{m} \right)\left( 1+\frac{1}{j^{m+1}}f_{m+1} \right) F_{j}
\\
	= 1 + \left[ f_{m+1}+C_{m+1} \right]j^{-m-1} + {\rm O}(j^{-m-2}) .
\end{multline*}
Now, as before, we choose $ f_{m+1} = -C_{m+1} $, and the leading correction is now ${\rm O}(j^{-m-2}) $ thus proving our claim.
\end{proof}

\begin{remark}
We will find subsequently in Prop.~\ref{fGrowth} that the product \eqref{large-j_Fj} does not converge as $m\to \infty$ because the growth of $ f_{m} $ is super-exponential 
(see \eqref{f2tratio}) so the representation is asymptotic.
There is a parallel expansion for large $ \gamma $ but in each factor both the numerator and denominator contain $j$ and this has disadvantages for our subsequent applications.
\end{remark}

\begin{remark}
The combination $ (1-\alpha)^{m}f_{m} $ is a polynomial in $\alpha$ whose leading and trailing coefficients in respect of $ \alpha $ are given by
\begin{multline}
	(1-\alpha)^{m}\gamma^{m-1}f_{m} = 
	(-1)^{m-1} \alpha \Big[ \left( (1+\gamma)^{m-1}-\gamma^{m-1} \right)\alpha^{m-2} + \ldots
\\
		 + \left( \tfrac{1}{2}(2^{m}-2m)+\tfrac{1}{4}(2^{m+1}-2m-3-(-1)^{m})\gamma \right)\alpha + 1 \Big] .
\end{multline}
The list of the first six $f_{n}$ are:
\begin{gather}
	(1-\alpha)f_1 = \alpha ,
\label{f1}
\\
	(1-\alpha)^2 \gamma f_2 = -\alpha ,
\label{f2}
\\
	(1-\alpha)^3 \gamma^2 f_3 = \alpha  [ \alpha  (2 \gamma +1)+1 ] ,
\label{f3}
\\
	(1-\alpha)^4 \gamma^3f_4 = -\alpha  \left[ \alpha ^2 \left(3 \gamma^2+3 \gamma +1\right)+\alpha  (5 \gamma +4)+1 \right] ,
\label{f4}
\end{gather}
\begin{multline}
	(1-\alpha )^5 \gamma^4 f_5 = 
\\
	\alpha
	\left[ \alpha ^3 \left(4 \gamma^3+6 \gamma^2+4 \gamma +1\right)+\alpha ^2 \left(16 \gamma^2+25 \gamma +11\right)+\alpha  (13 \gamma +11)+1 \right] ,
\label{f5}
\end{multline}
and
\begin{multline}
	(1-\alpha)^6 \gamma^5 f_6 =
\\
	 -\alpha  
	\left[
		\alpha ^4 \left(5 \gamma^4+10 \gamma^3+10 \gamma^2+5 \gamma +1\right)+\alpha ^3 \left(33 \gamma^3+83 \gamma^2+78 \gamma +26\right)
	\right.
\\		
	\left.
		+\alpha ^2 \left(65 \gamma^2+129 \gamma +66\right)+2 \alpha  (14 \gamma +13)+1
	\right] .
\label{f6}
\end{multline}
\end{remark}


Prior to continuing we require a preliminary identity for an infinite sum to be recast as a finite sum. 
\begin{lemma}\label{PowerSum2StirlingSum}
Let $ m\in \Z_{\geq 0} $, $ 0<\alpha<1 $ and $ \gamma>0 $. Then the following identity holds:
\begin{equation}
	\sum_{k=0}^{\infty}\alpha^{k}\frac{(\gamma)_{k}}{k!}k^{m} = (1-\alpha)^{-\gamma}\sum_{l=0}^{m}S(m,l)(\gamma)_{l}\left( \frac{\alpha}{1-\alpha} \right)^{l} .
\label{PowerSum}
\end{equation}
where $ S(m,i) = \left\{ {m\atop i} \right\} $ is the Stirling number of the second kind or the Stirling "partition" number in the notation of \cite[\S 26.8(i)]{DLMF},
i.e. the number of partitions of ${1,2,\ldots,m}$ into exactly $i$ non-empty subsets.
\end{lemma}
\begin{proof}
The Stirling numbers of the second kind are the matrix elements of the transformation matrix giving the monomial basis in terms of the descending factorial basis,
see \cite[\S 26.8(i)]{DLMF}, which is expressed as
\begin{equation*}
	k^m = \sum_{l=0}^{m}S(m,l)(k-l+1)_{l} = \sum_{l=0}^{m}S(m,l)\frac{k!}{(k-l)!} . 
\end{equation*}
Employing this in the left-hand side and interchanging the summations as the $k$-sum is uniformly and absolutely convergent
\begin{align*}
	\sum_{k=0}^{\infty}\alpha^{k}\frac{(\gamma)_{k}}{k!}k^{m} 
	& = \sum_{l=0}^{m} S(m,l) \sum_{k=l}^{\infty}\alpha^{k}\frac{(\gamma)_{k}}{(k-l)!} ,
\\
	& = \sum_{l=0}^{m} S(m,l) \sum_{i=0}^{\infty}\alpha^{i+l}\frac{(\gamma)_{i+l}}{i!} ,
\\
	& = \sum_{l=0}^{m} S(m,l)\alpha^{l}(\gamma)_{l} \sum_{i=0}^{\infty}\alpha^{i}\frac{(\gamma+l)_{i}}{i!} ,\quad \text{using $(\gamma)_{i+l} = (\gamma)_{l}(\gamma+l)_{i}$} ,
\\
	& = \sum_{l=0}^{m} S(m,l)\alpha^{l}(\gamma)_{l} (1-\alpha)^{-\gamma-l} ,\quad \text{by the binomial theorem} ,
\\
	& = (1-\alpha)^{-\gamma}\sum_{l=0}^{m} S(m,l)(\gamma)_{l} \left(\frac{\alpha}{1-\alpha}\right)^{l} .
\end{align*}
\end{proof}

In the proof of Prop. \ref{ProductFormula} we only indicated how successive $f_{n}$ could be determined in principle but didn't give an explicit formula - 
our next result gives an explicit recursive means to compute the coefficients $ f_{m} $.
\begin{proposition}
Let $ 0\leq \alpha < 1 $, $ \gamma > 0 $.
The coefficients $ f_{m} $, $ m\geq 1 $ satisfy the recursive formula
\begin{multline}
	f_{m+1} = (-1)^{m}\gamma^{-m-1} \sum_{i=0}^{m+1} S(m+1,i)(\gamma)_{i}\left( \frac{\alpha}{1-\alpha} \right)^{i}
\\
	+ \sum_{n=0}^{m} (-1)^{n+1}\gamma^{-n} \sum_{i=0}^{n} S(n,i)(\gamma)_{i}\left( \frac{\alpha}{1-\alpha} \right)^{i}
			\times \sum_{\substack{m \geq l_{s} > \cdots > l_{1} \geq 1 \\ \sum_{k=1}^{s}l_{k} = m+1-n}} f_{l_{s}} \cdots f_{l_{1}} .
\label{f_recur}
\end{multline}
Therefore $ (1-\alpha)^{m+1}\gamma^{m}f_{m+1} $ is a polynomial in $ \alpha, \gamma $ of degrees $m, m-1$ respectively.
The initial coefficient $ f_1 $ is given by the first term of the right-hand side at $m=0$, the second term having an empty summand.
\end{proposition}
\begin{proof}
In order to determine $ f_{m+1} $ we only need to extract the coefficient of $ j^{-m-1} $ in 
\begin{equation*}
	(1-\alpha)^{\gamma} \left( 1+\frac{1}{j}f_{1} \right)\cdots \left( 1+\frac{1}{j^m}f_{m} \right)F_{j} .
\end{equation*}
In addition we are going to split the sum for $F_{j}$ \eqref{F-ParFrac} using the following
\begin{equation*}
	\frac{j}{j+k\gamma^{-1}} = \sum_{n=0}^{m+1}(-1)^{n}\left(\frac{k}{j\gamma}\right)^{n} + (-1)^{m+2}\left(\frac{k}{j\gamma}\right)^{m+2}\frac{j}{j+k\gamma^{-1}} .
\end{equation*}
Thus we have
\begin{align*}
	-f_{m+1}
	& = [j^{-m-1}](1-\alpha)^{\gamma}\sum_{k\geq 0}\alpha^{k}\frac{(\gamma)_{k}}{k!}\sum_{n=0}^{m+1}\left( 1+\frac{1}{j}f_{1} \right) \cdots \left( 1+\frac{1}{j^m}f_{m} \right)
		(-1)^{n}\left(\frac{k}{j\gamma}\right)^{n} ,
\\
	& = [j^{-m-1}](1-\alpha)^{\gamma}\sum_{k\geq 0}\alpha^{k}\frac{(\gamma)_{k}}{k!}\sum_{n=0}^{m}\left( 1+\frac{1}{j}f_{1} \right) \cdots \left( 1+\frac{1}{j^m}f_{m} \right)
		 (-1)^{n}\left(\frac{k}{j\gamma}\right)^{n}
\\
	& \qquad
	 + (1-\alpha)^{\gamma}\sum_{k\geq 0}\alpha^{k}\frac{(\gamma)_{k}}{k!}(-1)^{m+1}\left(\frac{k}{\gamma}\right)^{m+1} ,
\\
	& = (-1)^{m+1}\gamma^{-m-1}(1-\alpha)^{\gamma}\sum_{k\geq 0}\alpha^{k}\frac{(\gamma)_{k}}{k!}k^{m+1}
\\
	& 
		+ [j^{-m-1}](1-\alpha)^{\gamma}\sum_{k\geq 0}\alpha^{k}\frac{(\gamma)_{k}}{k!}\sum_{n=0}^{m}(-1)^{n}\left(\frac{k}{\gamma}\right)^{n}j^{-n}
			\sum_{\substack{m \geq l_{s} > \cdots > l_{1} \geq 1}} j^{-l_{1}}f_{l_{1}} \cdots j^{-l_{s}}f_{l_{s}} ,
\\
	& = (-1)^{m+1}\gamma^{-m-1}(1-\alpha)^{\gamma}\sum_{k\geq 0}\alpha^{k}\frac{(\gamma)_{k}}{k!}k^{m+1}
\\
	& \qquad
		+ (1-\alpha)^{\gamma}\sum_{k\geq 0}\alpha^{k}\frac{(\gamma)_{k}}{k!}\sum_{n=0}^{m}(-1)^{n}\left(\frac{k}{\gamma}\right)^{n}
			\sum_{\substack{m \geq l_{s} > \cdots > l_{1} \geq 1 \\ \sum_{i=1}^{s}l_{i} = m+1-n}} f_{l_{1}} \cdots f_{l_{s}} .
\end{align*}
Applying identity \eqref{PowerSum} to the $k$-sums appearing above then yields \eqref{f_recur}.
\end{proof}

Let us define a $t$-sequence which contains all the $\alpha$, $\gamma$ dependencies that occur in the recursive formula \eqref{f_recur}, for $n\geq 0 $,
\begin{equation}
	t_{n}(\alpha,\gamma) := (-1)^{n}\gamma^{-n} \sum_{i=0}^{n} S(n,i)(\gamma)_{i} \left( \frac{\alpha}{1-\alpha} \right)^{i} .
\label{tDefn}
\end{equation}
We note for standard conditions, i.e. $ 0< \alpha <1 $, $ \gamma>0 $, and by the combinatorial interpretation of $ S(n,i) $ that $ t_n $ oscillates with
period $2$ in $n$ being positive for even $n$.
The first six $t_{n}$ are: $ t_{-1}=0 $, $ t_0 = 1 $,
\begin{gather}
	(1-\alpha) t_1 =  -\alpha ,
\label{t1}
\\
	(1-\alpha)^2 \gamma t_2 =  \left[ \alpha ^2 \gamma +\alpha \right] ,
\label{t2}
\\
	(1-\alpha)^3 \gamma^2 t_3 =  -\left[ \alpha ^3 \gamma ^2+\alpha ^2 (3 \gamma +1)+\alpha \right] ,
\label{t3}
\\
	(1-\alpha)^4 \gamma^3 t_4 =  \left[ \alpha ^4 \gamma ^3+\alpha ^3 \left(6 \gamma ^2+4 \gamma +1\right)+\alpha ^2 (7 \gamma +4)+\alpha \right] ,
\label{t4}
\end{gather}
\begin{multline}
	(1-\alpha)^5 \gamma^4 t_5 = 
\\
	 -\left[ \alpha ^5 \gamma ^4+\alpha ^4 \left(10 \gamma ^3+10 \gamma ^2+5 \gamma +1\right)+\alpha ^3 \left(25 \gamma ^2+30 \gamma +11\right)+\alpha ^2 (15 \gamma +11)+\alpha \right] ,
\label{t5}
\end{multline}
and
\begin{multline}
	(1-\alpha)^6 \gamma^5 t_6 =  
	\left[ 
	\alpha ^6 \gamma ^5+\alpha ^5 \left(15 \gamma ^4+20 \gamma ^3+15 \gamma ^2+6 \gamma +1\right)
	\right.
\\	\left.
	+\alpha ^4 \left(65 \gamma ^3+120 \gamma ^2+91 \gamma +26\right)+\alpha ^3 \left(90 \gamma ^2+146 \gamma +66\right)+\alpha ^2 (31 \gamma +26)+\alpha
	\right] .
\label{t6}
\end{multline}

Then the recursive formula \eqref{f_recur} takes the generic form
\begin{equation}
	f_{m+1} = -t_{m+1}
	- \sum_{n=0}^{m} t_{n} \sum_{\substack{m \geq l_{s} > \cdots > l_{1} \geq 1 \\ \sum_{k=1}^{s}l_{k} = m+1-n}} f_{l_{s}} \cdots f_{l_{1}} .
\label{f_RECUR}
\end{equation}
A unified formula can be written for\eqref{f_RECUR} by allowing for a zero index $l=0$ on the $f_l$ along with $f_0:=1$ so that we have
\begin{equation}
	\sum_{n=0}^{m+1} t_{n} \sum_{\substack{m+1 \geq l_{s} > \cdots > l_{1} \geq 0 \\ n+\sum_{k=1}^{s}l_{k} = m+1}} f_{l_{s}} \cdots f_{l_{1}} = 0 .
\label{t-f_relation}
\end{equation}
Now the $t_{m+1}$ term corresponds to $n=m+1$ and the $f_{m+1}$ term to $n=0$ in the expanded inner summation.
Consequently we have the expansion
\begin{equation}
	(1-\alpha)^{\gamma}F_{j} = \sum_{k\geq 0} j^{-k}t_{k}(\alpha,\gamma) ,
\label{AsypmtoticFj}
\end{equation}
and will defer the question of convergence or otherwise of this sum until \S \ref{Bounds+Growth}.

\section{Power Product Expansions and Lambert Series}\label{PPE}
At this point it is advantageous to step back from the specific problem at hand and recognise a larger picture -
what we have been doing is constructing a {\it power product expansion} (PPE) for a function that is given in terms of an analytic expansion about $x=0$,
i.e. in terms of an infinite product with a particular structure of its factors. 
For some background on PPEs one should consult \cite{GGM_1988}, \cite{GK_1995}, \cite{Alk_2009}, \cite{GQ_2014}.
We define this via the relation 
\begin{equation}
	A(x) := 1+\sum_{n\geq 1}a_{n}x^{n} = \prod_{n\geq 1} \left( 1+f_{n}x^{n} \right) ,
\label{AGF}
\end{equation}
where the left-hand side encodes a sequence $\{a_{n}\}_{n\geq 1}$ and the right-hand side another sequence $\{f_{n}\}_{n\geq 1}$.
These sums and products may be convergent or be purely formal as generating functions.
On the right-hand side we see an obvious combinatorial interpretation as partitions without repetitions, or strict partitions.
In the traditional combinatorial approach there is a "simple" expression for the $a$-coefficient, so that for $m\geq 1$ one has
\begin{equation}
	a_{m}  = \sum_{\substack{m \geq l_{s} > \cdots > l_{1} \geq 1 \\ \sum_{k=1}^{s}l_{k} = m}} f_{l_{s}} \cdots f_{l_{1}} .
\label{f-to-a:b}
\end{equation}
The first six $a_{n}$ in terms of $f_{n}$ are:
\begin{gather}
	a_{1} = f_{1} ,
\label{a1}
\\
	a_{2} = f_{2} ,
\label{a2}
\\
	a_{3} = f_3 + f_2 f_1 ,
\label{a3}
\\
	a_{4} = f_4 + f_3 f_1 ,
\label{a4}
\\
	a_{5} = f_5 + f_4 f_1 + f_3 f_2 ,
\label{a5}
\end{gather}
and
\begin{equation}
	a_{6} = f_6 + f_5 f_1 + f_4 f_2 + f_3 f_2 f_1 .
\label{a6}
\end{equation}

However, in contrast to many studies of additive combinatorics, 
we are not going to study examples where the $f_{n}$ are given with simple forms and the problem is to determine the consequent properties of the $a_{n}$, 
which have the combinatorial interpretation, but to do the reverse. 
Thus we will take the $a_{n}$ as given and wish to construct the $f_{n}$ in terms of the former or study their properties.
In this context the $a_{n}$ are indeterminates, whether or not they have any explicit form.
It appears that there is not an equivalent "simple" expression for the inverse situation.
Inverting the relation \eqref{f-to-a:b} for the first members gives
\begin{gather}
	f_1 = a_1 ,
\label{f1<-a}
\\
	f_2 = a_2 ,
\label{f2<-a}
\\
	f_3 = a_3 - a_2 a_1 ,
\label{f3<-a}
\\
	f_4 =   a_4 - a_3 a_1 + a_2 a_1^2 ,
\label{f4<-a}
\\
	f_5 =  a_5 - a_4 a_1 - a_3 a_2 + a_3 a_1^2 + a_2^2 a_1 - a_1^3 a_2 ,
\label{f5<-a}
\end{gather}
and
\begin{equation}
	f_6 = a_2 a_1^4 - a_3 a_1^3 - a_2^2 a_1^2 + a_4 a_1^2 + a_2 a_3 a_1 - a_5 a_1 - a_2 a_4 + a_6 .
\label{f6<-a}
\end{equation}
See Comtet \cite{Comtet_1974}, Chap. II, Supp. \& Ex. 16, pg. 120-1; Eq. (10) of \cite{GGM_1988}; pg. 1221 of \cite{GK_1995}; pg. 94-5 of \cite{GQ_2020}.
A related thread closely connected to this problem is the study of generalised Lambert series and their arithmetic properties, 
see \cite{Bor_1931}, \cite{Bor_1931a}, \cite{Fel_1932}, \cite{Gar_1936}, \cite{Doy_1939}.
On the applications of {M}\"{o}bius inversion in combinatorial analysis see \cite{BG_1975}.
In addition many of these studies have treated a generalised power product expansion (GPPE) \cite{GQ_2014} of the form
\begin{equation*}
	1+\sum_{n\geq 1}a_{n}x^{n} = \prod_{n\geq 1} \left( 1+f_{n}x^{n} \right)^{r_{n}} .
\label{GPPE}
\end{equation*}

Let us recall the generating function definitions \eqref{AGF} and form its logarithmic derivative
\begin{equation}
	x\frac{A'}{A} := \sum_{n=1}^{\infty} d_{n}x^{n} = \sum_{n=1}^{\infty} \frac{nf_{n}x^{n}}{1+f_{n}x^{n}},
\label{dGF}
\end{equation}
which introduces the Lambert series, see Knopp, Chap. XII, \S 58 C \cite{Knopp_1990}.
Let us write the partition of $ n $ as a sum of multiplicities 
$ \lambda(n) \equiv \left( 1^{\lambda_{1}}2^{\lambda_{2}} \cdots \right)\vdash n = 1\,\lambda_{1}+2\,\lambda_{2}+ \cdots +n\,\lambda_{n} $
(pg. 1  of \cite{Andrews_1998} for an explanation of the notation).
Gingold and Knopfmacher \cite{GK_1995} deduced the form of $ f_{n} $'s in terms of $ a_{n} $'s in Lemma 2.1, Eq. (2.3),
\begin{equation*}
	f_{n} = \sum_{\lambda \vdash n} c(\lambda)a_{1}^{\lambda_{1}} \cdots a_{n}^{\lambda_{n}} ,
\end{equation*}
without giving a general, closed formula for the coefficients $c(\lambda)$.
They investigated the absolute sum of these coefficients
\begin{equation*}
	B(n) := \sum_{\lambda(n)} |c(\lambda)| ,
\end{equation*}
and this sequence is recorded in OEIS \cite{OEIS} as {\tt A220418}.
They found some properties and in particular bounds for $ B(n) $,
\begin{equation*}
	\frac{2^{n-1}}{n} \leq B(n) < \frac{2^{n}}{n} , \quad n\geq 1 ,
\end{equation*}
along with the asymptotic growth estimates
\begin{equation*}
	B(n) \mathop{=}\limits_{n \to \infty} \frac{2^{n}}{n}\left( 1+{\rm O}(n^{-1}) \right) .
\end{equation*}

\subsection{$t_{n}$, $a_{n}$ and $d_{n}$-Coefficients}

The $t_{n}$ and $a_{n}$ coefficients are related by a linear relation for all $m\geq 0$,
\begin{equation}
	-t_{m+1} = t_{m}a_{1} + t_{m-1}a_{2} + \ldots + t_{1}a_{m} + t_{0}a_{m+1} .
\label{t<->a}
\end{equation}
This follows from comparing \eqref{t-f_relation} with \eqref{f-to-a:b} from which one can conclude
\begin{equation*}
	\sum_{l=0}^{m+1}t_{l}a_{m+1-l} = 0 ,
\end{equation*}
taking $ a_{0} = 1 $.
The above relation can be formulated as a matrix equation with a lower triangular Toeplitz structure whose solution for the $a_{n}$ coefficients is given by
\begin{equation}
	\begin{pmatrix}
	a_{1} \\ a_{2} \\ a_{3} \\ a_{4} \\ \vdots
	\end{pmatrix}
	= -
	\begin{pmatrix}
		t_{0} 	& 0 		& 0 		& 0 		& 0 & \cdots \\
		t_{1} 	& t_{0} 	& 0 		& 0 		& 0 & \cdots \\
		t_{2} 	& t_{1} 	& t_{0} 	& 0 		& 0 & \cdots \\	
		t_{3} 	& t_{2} 	& t_{1} 	& t_{0} 	& 0 & \cdots \\
		\vdots 			& \vdots 		& \vdots 		& \vdots 	& \vdots 			
	\end{pmatrix}^{-1}
	\begin{pmatrix}
	t_{1} \\ t_{2} \\ t_{3} \\ t_{4} \\ \vdots
	\end{pmatrix} .	
\end{equation}

Another formula for $a_{n}$, $n\geq 0$, can be found via the geometrical generating functions
\begin{align}
	a_{n} = -[x^{n}] \frac{\sum_{j=1}^{\infty} t_{j}x^{j}}{1+\sum_{j=1}^{\infty} t_{j}x^{j}} .
\label{t->a}
\end{align}
\begin{proposition}
We have
\begin{equation}
	a_{n} = \sum_{\substack{m_{l} \geq 0 \\ \sum_{l\geq 1}l m_{l} = n}} (-1)^{\sum_{l\geq 1}m_{l}} {{\sum_{l\geq 1}m_{l}}\choose{m_{1},m_{2},\ldots}} \prod_{l\geq 1}t^{m_{l}}_{l} ,
	\quad n\geq 1 ,
\label{a<-tProduct}
\end{equation}
or alternatively
\begin{equation}
	A(x) = \left( 1 + \sum_{l=1}^{\infty}t_{l}x^{l} \right)^{-1} .
\label{AGF<-tGF}
\end{equation}
\end{proposition}
\begin{proof}
Recalling \eqref{AGF}, relation \eqref{AGF<-tGF} follows directly from \eqref{t->a}.
Without loss of generality we can take $b_0$ to be unity in a series expansion of the reciprocal
\begin{equation}
	\left( 1+\sum_{n=1}^{\infty} b_{n}x^{n} \right)^{-1} =
	\sum_{n=0}^{\infty} x^{n} 
	\sum_{\substack{m_{l} \geq 0 \\ \sum_{l=1}l m_{l} = n}} (-1)^{\sum_{l}m_{l}} {{\sum_{l\geq 1}m_{l}}\choose{m_{1},m_{2},\ldots}} \prod_{l\geq 1}b^{m_{l}}_{l} ,	
\end{equation}
using the multinomial expansion, see \cite[\S 3.5]{Comtet_1974}, \cite{WW_2009}.
\end{proof}
Explicit solutions for low coefficients are:
\begin{gather}
	a_{1} = -t_1 ,
\label{a1<-t}
\\
	a_{2} = -t_2+t_1^2 ,
\label{a2<-t}
\\
	a_{3} = -t_3+2 t_2 t_1-t_1^3 ,
\label{a3<-t}
\\
	a_{4} = -t_4+2 t_3 t_1+t_2^2-3 t_2 t_1^2+t_1^4 ,
\label{a4<-t}
\\
	a_{5} =	-t_5+2 t_4 t_1+2 t_3 t_2-3 t_3 t_1^2-3 t_2^2 t_1+4 t_2 t_1^3-t_1^5 ,
\label{a5<-t}
\end{gather}
and
\begin{equation}
	a_{6} =	-t_6+2 t_5 t_1+2 t_4 t_2-3 t_4 t_1^2+t_3^2-6 t_3 t_2 t_1+4 t_3 t_1^3-t_2^3+6 t_2^2 t_1^2-5 t_2 t_1^4+t_1^6 .
\label{a6<-t}
\end{equation}
Note that the foregoing relations are completely reflexive, i.e. symmetrical under $ t_{n}\leftrightarrow a_{n} $, as one can observe from identity \eqref{AGF<-tGF}.

Recalling \eqref{dGF}, we can deduce a similar linear recurrence relation linking the $d_{n}$ coefficients to the $t_{m}$ coefficients.
\begin{proposition}
The $d_{n}$ coefficients and the $t_{m}$ coefficients satisfy the convolution relation
\begin{equation}
	-m t_m = \sum_{n=1}^{m} t_{m-n}d_{n} ,\quad m\geq 1,
\label{t<->d}
\end{equation}
which can be written as a recurrence relation for the $d_{n}$ coefficients in the forward direction
\begin{equation}
	d_{m} = - d_{m-1}t_{1} - \ldots - d_{1}t_{m-1} - mt_{m} ,\quad d_0=0, m\geq 1 .
\label{d-Recur}
\end{equation}  
\end{proposition}
\begin{proof}
It follows from \eqref{dGF} and \eqref{AGF<-tGF} that
\begin{equation*}
	\sum_{n\geq 1}d_{n}x^{n} = x\frac{d}{dx}\log A(x) = -x\frac{d}{dx}\log \bigg( 1+\sum_{l\geq 1}t_{l}x^{l} \bigg) ,
\end{equation*}
whence \eqref{d-Recur}.
\end{proof}

In addition the $d_{n}$ coefficients are related to the $a_{n}$ coefficients by the simple convolution relation
\begin{equation}
	n a_{n} = d_{n} + \sum_{k=1}^{n-1}d_{k}a_{n-k} ,\quad n\geq 1,
\label{a<->d}
\end{equation}
which follows from the definition.

\subsection{$f_{n}$-Coefficients}

The first six $f_{n}$ in terms of the $t_{n}$ are:
\begin{gather}
	f_1 = -t_1 ,
\label{f1<-t}
\\
	f_2 = t_1^2-t_2 ,
\label{f2<-t}
\\
	f_3 = t_1 t_2-t_3 ,
\label{f3<-t}
\\
	f_4 = t_1^4-2 t_2 t_1^2+t_3 t_1+t_2^2-t_4 ,
\label{f4<-t}
\\
	f_5 = t_2 t_1^3-t_3 t_1^2-t_2^2 t_1+t_4 t_1+t_2 t_3-t_5 ,
\label{f5<-t}
\end{gather}
and
\begin{equation}
	f_6 = t_3 t_1^3+t_2^2 t_1^2-t_4 t_1^2-3 t_2 t_3 t_1+t_5 t_1+t_3^2+t_2 t_4-t_6 .
\label{f6<-t}
\end{equation}

A fundamental relation derived by many authors \cite{Rit_1930}, \cite{Fel_1932}, \cite{GK_1995} relates the $f_{n}$ and the $d_{n}$,
with a number of these studies treating the more general setting of the GPPE.
\begin{proposition}[\cite{Rit_1930}, \cite{Fel_1932}, Eq. (2.12)\cite{GK_1995}]
The $f_{n}$ and the $d_{n}$ satisfy the arithmetic convolution relation
\begin{equation}
	f_{n} = \frac{1}{n}d_{n} + \sum_{\substack{d|n \\ d>1}} \frac{1}{d}\left( -f_{n/d} \right)^{d} .
\label{fCoeff<-dCoeff}
\end{equation}
\end{proposition}
\begin{proof}
The proof of this is well-known and we reproduce it here for the convenience of the reader.
Starting with \eqref{dGF} one computes
\begin{align*}
	\frac{A'(x)}{A(x)} := \sum_{n=1}^{\infty} d_{n}x^{n-1} & = \sum_{n=1}^{\infty} \frac{nf_{n}x^{n-1}}{1+f_{n}x^{n}} ,
\\
	& = \sum_{n=1}^{\infty} nf_{n}x^{n-1} \sum_{m=0}^{\infty} (-1)^{m}f_{n}^{m}x^{mn} ,
\\
	& = \sum_{n\geq 1,m\geq 0} (-1)^{m}nf_{n}^{m+1}x^{(m+1)n-1}, \quad l=(m+1)n,
\\
	& = - \sum_{l\geq 1} x^{l-1} \sum_{n|l,n\geq 1} n(-1)^{l/n}f_{n}^{l/n},
\\
	& = - \sum_{l\geq 1} x^{l-1} \sum_{d|l,d\geq 1} (-1)^{d}\frac{l}{d}f_{l/d}^{d},
\end{align*}
and compares coefficients. Separating out the term with $d=1$ gives \eqref{fCoeff<-dCoeff}.
\end{proof}

\begin{remark}
Note that \eqref{fCoeff<-dCoeff} is not a Dirichlet convolution, so some of the powerful analytical number theory tools are not available here.
\end{remark}

\begin{remark}
Ritt \cite{Rit_1930}, Feld \cite{Fel_1932} and subsequent authors have undertaken studies of the growth of the $ f_{n} $ coefficients and the consequent convergence of the associated power-product expansion,
albeit in a more general structure, but with the assumption that $ r = \mathop{\limsup}\limits_{n \geq 1} |d_{n}|^{1/n} $ exists.
However as we will see this does not apply here, and in our case the rate of growth is faster.
\end{remark}

\section{Bounds and Growth Estimates}\label{Bounds+Growth}

We seek estimates and bounds on the growth of $ f_{m}(\alpha,\gamma) $ as $ m\to\infty $ for all $ 0\leq\alpha<1 $ and bounded $ \gamma>0 $.
We will undertake this task in a step-by-step manner, firstly for $ t_{n} $, then for $ d_{n} $ and finally $ f_{n} $.
Let us begin by recalling the definition \eqref{tDefn} of the $t_{n}$-sequence which contains all the $\alpha$, $\gamma$ dependencies that occur in the recursive formula \eqref{f_recur}, 
for $n\geq 0 $.

\begin{remark}
In the defining sum \eqref{tDefn} we can write the Pochhammer symbol as a descending factorial $ (-x)_{i} = (-1)^{i}x(x-1) \cdots (x-i+1) $ 
and employ the other defining characteristic of the second kind Stirling numbers as the coefficients of descending factorial expansion of the monomials (see Eq. 26.8.10 \cite{DLMF}),
\begin{equation*}
	\sum_{i=1}^{n} S(n,i)(-1)^{i}(-x)_{i} = \sum_{i=1}^{n} S(n,i) x(x-1) \cdots (x-i+1) = x^{n} .
\end{equation*}
This allows us to deduce the limiting identity $ t_{n}(\alpha,-\gamma) \mathop{\rightarrow}\limits_{\alpha \to \infty} (-1)^n $, 
which is obvious from the explicit formulae  \eqref{t1}--\eqref{t6}.
\end{remark}

The $t_n$ coefficients possess an exponential generating function, which we will subsequently show to be convergent about $ z=0 $ (see Prop. \ref{tUpperBounds}),
\begin{equation}
	\sum_{n=0}^{\infty} \frac{z^{n}}{n!} t_{n} = (1-\alpha)^{\gamma} \left[ 1-\alpha e^{-z/\gamma} \right]^{-\gamma} ,
\label{tExpGenFun}
\end{equation}
so that
\begin{equation}
	t_{n} = (-1)^{n}\left. (1-\alpha)^{\gamma}\gamma^{-n} \frac{d^{n}}{dz^{n}}\left[ 1-\alpha e^{z} \right]^{-\gamma}\right|_{z=0} ,
\label{tDeriv}
\end{equation}
or alternatively
\begin{equation}
	t_{n} = (-1)^{n} (1-\alpha)^{\gamma}\gamma^{-n} \frac{n!}{2\pi i} \oint_{C_0} \frac{dz}{z^{n+1}}\left[ 1-\alpha e^{z} \right]^{-\gamma} ,
\label{tLoopInt}
\end{equation}
where $ C_0 $ is a loop contour enclosing the origin and none of the branch points $ z_m = -\log\alpha + 2\pi im $, $ m\in \Z $.
The first result follows from the exponential generating function for the Stirling numbers of the second kind in \cite{DLMF}, Eq. 26.8.12. 
In addition these coefficients possess a purely formal geometrical generating function
\begin{equation}
	\sum_{n=0}^{\infty} z^{n} t_{n} = (1-\alpha)^{\gamma} {}_2F_1(\gamma,\gamma z^{-1};1+\gamma z^{-1};\alpha)
	= (1-\alpha)^{\gamma} \sum_{l\geq 0}\alpha^{l}\frac{(\gamma)_{l}}{l!}\frac{\gamma}{\gamma+lz} ,
\label{tGeomGenFun}
\end{equation} 
which is evident from the sequence of poles at $ z=-\gamma/m $, $ m\in \N $ accumulating at $ z=0 $.
This latter result can be deduced from \cite{DLMF}, Eq. 26.8.11, but such a deduction involves violating a technical condition known to apply here.

Next we give an identity which is the same as in Lemma \ref{PowerSum2StirlingSum}, but we reverse the logic that was employed earlier.
\begin{proposition}
The $t_n$ coefficients are given by the convergent expansion about $ \alpha = 0 $, $ |\alpha|<1 $:
\begin{equation}
	t_{n}(\alpha,\gamma) = (-1)^{n}\gamma^{-n}(1-\alpha)^{\gamma} \sum_{l=0}^{\infty} l^{n}\frac{(\gamma)_{l}}{l!}\alpha^{l} .
\label{tSeries}
\end{equation}
\end{proposition}
\begin{proof}
We choose a circular loop centred on the origin for $C_0 $, $z=re^{i\theta}$, with $ 0<r<\log\alpha^{-1}$.
Then we expand the radical factor of the integral \eqref{tLoopInt} via the binomial theorem arriving at
\begin{equation*}
	t_{n} = (-1)^{n}\gamma^{-n}(1-\alpha)^{\gamma}n! r^{-n} \sum_{l=0}^{\infty}\alpha^{l}\frac{(\gamma)_{l}}{l!}
			\int_{-\pi}^{\pi}\frac{d\theta}{2\pi}e^{-in\theta}\exp(lre^{i\theta}) .
\end{equation*}
The angular integral is easily evaluated by series expanding the exponential in the last factor in the integrand and we find this integral is $(lr)^{n}/n!$.
This facilitates the cancelling of all the $r$-dependence, and we have \eqref{tSeries}.
The ratio test then gives us the radius of convergence, as stated.
\end{proof}

\begin{proposition}
The $ t_{n}(\alpha,\gamma) $ coefficients satisfy the mixed difference equation
\begin{equation}
	t_{n}(\alpha,\gamma) = t_{n-1}(\alpha,\gamma) - \frac{1}{1-\alpha}\left( 1+\gamma^{-1} \right)^{n-1}t_{n-1}(\alpha,\gamma+1) ,
\label{tRecur}
\end{equation} 
or, in terms of the re-defined coefficient $ \tilde{t}_{n}(\alpha,\gamma) = \gamma^{n}t_{n}(\alpha,\gamma) $,
\begin{equation}
	\tilde{t}_{n}(\alpha,\gamma) = \gamma\tilde{t}_{n-1}(\alpha,\gamma) - \frac{\gamma}{1-\alpha}\tilde{t}_{n-1}(\alpha,\gamma+1) ,
\label{tAltRecur}
\end{equation}
\end{proposition}
\begin{proof}
The Stirling numbers of the second kind satisfy the recurrence relation, see Eq. 26.8.22 \cite{DLMF},
\begin{equation*}
	S(n,i) = iS(n-1,i)+S(n-1,i-1), \quad 0<i\leq n-1, n>1 .
\end{equation*}
Employing this in \eqref{tDefn} we have two terms on the right-hand side: 
in the first of these we write $ i=\gamma+i-\gamma $ so that $ i(\gamma)_{i}=(\gamma)_{i+1}-\gamma(\gamma)_{i} = \gamma(\gamma+1)_i-\gamma(\gamma)_{i} $ ;
in the second we replace $ i \mapsto i+1 $ and use $ (\gamma)_{i+1}=\gamma(\gamma+1)_{i} $ again.
As a result we have three terms which are of the same form as the definition \eqref{tDefn} with additional factors, and two of these coalesce. 
Simplifying we have \eqref{tRecur}.
\end{proof}

Central to our task will be knowledge of the growth of the $ t_{n} $ coefficients with respect to $n$, 
which will be obtained in two independent ways - via explicit upper and lower bounds for the magnitude of $ t_{n} $ and secondly via asymptotic estimates of the value of $ t_{n} $.
Our bounds will in fact be sharp and both will converge to our estimates.

One can find upper and lower bounds to the magnitude of the $t_n$ coefficients with a very similar form. 
\begin{proposition}\label{tLUBounds}
Let $ 0<\gamma<1 $ and $ 0<\alpha<1 $.
The magnitude of the $ t_{n} $ coefficients are bounded above through the formula 
\begin{equation}
	|t_{n}| < (1-\alpha)^{\gamma} \frac{\gamma^{-n}}{\Gamma(\gamma)}{\rm Li}_{-n-\gamma+1}(\alpha) ,
\label{tUpperBound}
\end{equation}
where $ {\rm Li}_{s}(z) $ is the polylogarithm as defined by \cite{DLMF}, Eq. 25.12.10.
The magnitude of the $ t_{n} $ coefficients are bounded below through the formula 
\begin{equation}
	|t_{n}| > (1-\alpha)^{\gamma} \frac{\gamma^{-n}}{\Gamma(\gamma)}
	\left[ {\rm Li}_{-n-\gamma+1}(\alpha)-(1-\gamma){\rm Li}_{-n-\gamma+2}(\alpha) \right] .
\label{tLowerBound}
\end{equation}
\end{proposition}
\begin{proof}
Our strategy is to apply bounds for the ratio $ (\gamma)_{l}/l! $ in formula \eqref{tSeries}. 
The simplest bounds suitable for this purpose is Gautschi's inequality, Eq. 5.6.4 of \cite{DLMF} or the original \cite{Gau_1959/60} which states, subject to $x>0$ and $0<s<1$,
\begin{equation*}
	(x+1)^{s-1} < \frac{\Gamma(x+s)}{\Gamma(x+1)} < x^{s-1} .
\end{equation*}
For the upper bound we use
\begin{equation*}
	\frac{\Gamma(l+\gamma)}{\Gamma(l+1)} < l^{\gamma-1} ,
\end{equation*}
and the definition of the polylogarithm.
For the lower bound we deduce a shifted identity derived from Gautschi's lower bound, namely
\begin{equation*}
	\frac{\Gamma(l+\gamma)}{\Gamma(l+1)} > (l+\gamma-1)l^{\gamma-2} ,
\end{equation*}
which introduces an additional, negative term.
\end{proof}

\begin{remark}
For the asymptotic forms of the bounds in Prop.~\ref{tLUBounds} we require polylogarithms with large, negative order and the following series representation suffices for this.
Let $ |\mu|<2\pi $ and $ s\neq 1,2,3,\ldots $ then we have \cite[Eq. 25.12.12]{DLMF}
\begin{equation*}
	{\rm Li}_{s}(e^{\mu}) = \Gamma(1-s)(-\mu)^{s-1} + \sum_{k=0}^{\infty}\frac{\zeta(s-k)}{k!}\mu^{k} ,
\end{equation*}
where $\zeta(s)$ is the Riemann zeta function, analytically continued to $ {\rm Re}(s)<0 $.
If the large $n$ asymptotics of the polylogarithm are employed then the above bounds reduce to
\begin{equation}
	|t_{n}| < (1-\alpha)^{\gamma} \frac{(\gamma)_{n}}{\gamma^{n}}\left( \log\alpha^{-1} \right)^{-n-\gamma} ,
\label{tUpperEst}
\end{equation}
and
\begin{equation}
	|t_{n}| > (1-\alpha)^{\gamma} \frac{(\gamma)_{n}}{\gamma^{n}}\left( \log\alpha^{-1} \right)^{-n-\gamma}
	\left[ 1 - \frac{(1-\gamma)}{n+\gamma-1}\log\alpha^{-1} \right] .
\label{tLowerEst}
\end{equation}
\end{remark}

In addition to the proof of Prop.~\ref{tLUBounds} there is an alternative approach to the upper bound.
\begin{proposition}\label{tUpperBounds}
The magnitude of the $ t_{n} $ coefficients are bounded above through the formula 
\begin{equation}
	|t_{n}| \leq (1-\alpha)^{\gamma} n! \left(\gamma r_0\right)^{-n-\gamma}\left(n+\gamma r_0\right)^{\gamma} ,
\label{tMinSup}
\end{equation}
where $ r_0 $ is given by
\begin{equation}
	r_0+\frac{n}{\gamma} = W_{0}\large( \tfrac{n}{\alpha\gamma}e^{n/\gamma}\large) ,
\label{r0Sup}
\end{equation}
with $ W_{0}(\cdot) $ being the principal branch of the Lambert $W$-function, \cite{DLMF} \S 4.13.
If the asymptotic formula for the Lambert $W_0$-function is again employed in the above equation we have
\begin{equation}
	\alpha e^{r_0} \mathop{\sim}\limits_{n \to \infty} \frac{n}{n-\gamma\log\alpha(1-\gamma n^{-1}+{\rm O}(n^{-2}))} .
\end{equation} 
\end{proposition}
\begin{proof}
Starting from \eqref{tLoopInt} specialised to a circular loop of radius $r$ one can deduce the upper bound
\begin{equation*}
	|t_{n}| \leq (1-\alpha)^{\gamma}\gamma^{-n}n! \sup_{z=r e^{i\theta},-\pi<\theta<\pi} r^{-n}\left| 1-\alpha e^{z} \right|^{-\gamma} .
\end{equation*}
The right-hand side has a supremum at $ \theta=0 $ of $ r^{-n}(1-\alpha e^{r})^{-\gamma} $.
This supremum can be minimised with respect to $r$ at $r_{0}$ given by
\begin{equation*}
	\alpha e^{r_{0}} = \frac{n}{n+\gamma r_{0}} ,
\end{equation*}
or by \eqref{r0Sup}. Clearly $ r_{0}< \log\alpha^{-1} $ as $ 1-\alpha e^{r_{0}} = \alpha\gamma n^{-1}r_{0}e^{r_{0}} $.
Furthermore, due to the elementary inequality $ (1+r)^{-1} > e^{-r} $ for $r>0$, we can deduce the lower bound $ n(n+\gamma)^{-1}\log\alpha^{-1} < r_{0} $.
The large $n$ asymptotics of $r_{0}$ follow from those of the 
$W_0$-Lambert function taken to third order\footnote{The author is indebted to a reviewer's insight that one has to take the expansion to third order to obtain
the correct sub-leading terms. If the expansion is terminated prematurely one gets contributions such as $\log n$ terms which end up being cancelled at the next order.}
in \cite[Eq. 4.13.1\_1]{DLMF}.
\end{proof}

\begin{proposition}
The $ t_{n} $ coefficients possess the large $n$ asymptotic growth formula
\begin{equation}
	t_{n} \mathop{\sim}\limits_{n \to \infty} (-1)^{n} \frac{(1-\alpha)^{\gamma}}{\sqrt{2\pi}} n! (n+1)^{-1/2} z_0^{-n-\gamma}\left(z_0+\frac{n+1}{\gamma}\right)^{\gamma}\left(z_0+\frac{n+1}{\gamma}+1\right)^{-1/2} ,
\label{tSteepDescent}
\end{equation}
where $ z_0 $ is given by
\begin{equation}
	z_0+\tfrac{n+1}{\gamma} = W_{0}\large( \tfrac{n+1}{\alpha\gamma}e^{(n+1)/\gamma}\large) ,
\label{tCritical}
\end{equation}
with $ W(\cdot) $ being the principal branch of the Lambert $W$-function, \cite{DLMF} \S 4.13.

If the asymptotic formula for the Lambert $W_0$-function is employed for $z_0 $ in the above equation, (see \cite{DLMF}, Eq. 4.13.1\_1), one has 
\begin{equation}
	t_{n} \mathop{\sim}\limits_{n \to \infty} (-1)^{n} \frac{1}{\sqrt{2\pi}}(1-\alpha)^{\gamma}\gamma^{1/2} n! (n+1)^{\gamma-1}\left(\gamma\log\alpha^{-1}\right)^{-n-\gamma} ,
\label{tAsymptotic}
\end{equation}
or the $n$-th root growth
\begin{equation}
	\left|t_{n}\right|^{1/n} \mathop{\sim}\limits_{n \to \infty} \frac{n}{e\gamma\log\alpha^{-1}} .
\label{tRootGrowth}
\end{equation}
A variant of the above estimate which is simpler for our purposes is the following formula
\begin{equation}
	t_{n} \mathop{\sim}\limits_{n \to \infty} (-1)^{n} (1-\alpha)^{\gamma} \frac{(\gamma)_{n}}{\gamma^{n}}\left( \log\alpha^{-1} \right)^{-n-\gamma} .
\label{tLargeOrder}
\end{equation}
\end{proposition}
\begin{proof}
We perform a steepest descent analysis of the integral \eqref{tLoopInt}.
An excellent reference for this method is \cite[Chap. 8]{Flajolet+Sedgewick_2009}.
The phase function is $ f(z) = -(n+1)\log z - \gamma\log(1-\alpha e^{z}) $ with a unique critical point denoted by $z_{0}$ and the solution of
\begin{equation*}
	\frac{n+1}{z_{0}} = \alpha\gamma\frac{e^{z_{0}}}{1-\alpha e^{z_{0}}} ,
\end{equation*}
which is solved by \eqref{tCritical}.
Employing the second derivative
\begin{equation}
	f_{0}^{''} = \frac{n+1}{z^2_{0}}\left[ 1+\frac{n+1}{\gamma}+z_{0} \right] ,
\end{equation}
we find the leading term to be given by \eqref{tSteepDescent}.
\end{proof}

The growth of the $ d_{n} $ coefficients with respect to $n$ is required and turns out that this is controlled by the growth of the $t_{m}$ coefficients -
the last term on the right-hand side of \eqref{d-Recur}, $mt_m$, dominates the other terms.
\begin{proposition}\label{dGrowth}
As $ m\to \infty $ $ d_{m} $ grows as
\begin{equation}
	-\frac{d_{m}}{mt_{m}} = 1+{\rm O}(m^{-1}) .
\label{d2tratio}
\end{equation}
\end{proposition}
\begin{proof}
Since $ t_{m} \neq 0 $ we can rewrite the convolution \eqref{d-Recur} as
\begin{equation*}
	\frac{d_{m}}{mt_{m}} = - 1 - \sum_{j=1}^{m-1} \frac{j}{m}\frac{t_{m-j}t_{j}}{t_{m}}\frac{d_{j}}{jt_{j}} ,
\end{equation*}
motivating the definition $ u_{m} := d_{m}/(mt_{m}) $. 
This allows us to deduce the inequality
\begin{equation*}
	|u_{m}|\leq 1 + \sum_{j=1}^{m-1} \frac{j}{m}\frac{|t_{m-j}||t_{j}|}{|t_{m}|}|u_{j}| .
\end{equation*}
This, in turn, facilitates the use of discrete analogues to the Gronwall-Bellman type inequalities \cite{Pac_1973}, \cite[Theorem 1.2.3]{Pachpatte_2002},
and given in their original formulation for the continuous setting \cite{Pea_1886}, \cite{Gro_1919}, \cite{Mik_1935} and \cite{Bel_1943}.
The result is the following explicit inequality for the $ u_{m} $ coefficients
\begin{equation*}
	|u_{m}| \leq 1 + \sum_{j=1}^{m-1} \frac{j}{m}\frac{|t_{m-j}||t_{j}|}{|t_{m}|} \prod_{k=j+1}^{m-1}\left( 1+\frac{k}{m}\frac{|t_{m-k}||t_{k}|}{|t_{m}|} \right) .
\end{equation*}
The relevant controlling aspect of the growth of $ u_{m} $ with $ m \to \infty $ is computed using the bounds \ref{tUpperBound} and \ref{tLowerBound},
\begin{align*}
 	\frac{|t_{m-k}||t_{k}|}{|t_{m}|} & \leq (1-\alpha)^{\gamma}\left( \log\alpha^{-1} \right)^{-\gamma}
 	\frac{(\gamma)_{m-k}(\gamma)_{k}}{(\gamma)_{m}}\left[ 1-\frac{1-\gamma}{m+\gamma-1}\log\alpha^{-1} \right]^{-1} ,
 \\
 	& = (1-\alpha)^{\gamma}\left( \log\alpha^{-1} \right)^{-\gamma} \frac{(\gamma)_{k}}{(m-1+\gamma)\cdots(m-k+\gamma)}\left[ 1+{\rm O}(m^{-1}) \right]
 \\
 	& = {\rm O}(m^{-k}) ,
\end{align*}
when $ k={\rm O}(1) $. As this ratio is symmetrical under $ k\mapsto m-k $ the same conclusion applies when $ k={\rm O}(m) $.
As is visible from the first line displayed above the ratio $|t_{m-k}||t_k|/|t_m|$ is essentially a reciprocal binomial coefficient 
and this decreases rapidly away from the $ k=0,m $ endpoints to a minimum around $ k=m/2 $.
This implies that the inner product is of order $ 1 + {\rm O}(m^{-j-1}) $ with $ j\geq 1 $ and consequently $ |u_{m}| \leq 1 + {\rm O}(m^{-1}) $.
\end{proof}

From the previous result we conclude that 
$ \limsup_{n\geq 1} n^{-1}|d_{n}|^{1/n} = \frac{1}{e\gamma\log\alpha^{-1}} $.
Our final estimates concern the growth of the $ f_{n} $ coefficients, and we observe that this in turn is completely controlled by the growth of the $ d_{n} $ coefficients. 
\begin{proposition}\label{fGrowth}
As $ n\to \infty $ $ f_{n} $ grows as
\begin{equation}
	\frac{nf_{n}}{d_{n}} = 1+{\rm O}(\kappa^n (\log n)^{\log 2}) ,
\label{f2dratio}
\end{equation}
where $ \kappa = 0.7514184\ldots $.
Consequently we have
\begin{equation}
	\lim_{n\to\infty}\frac{f_{n}}{t_{n}} = -1 , \quad \limsup_{n\geq 1} n^{-1}|f_{n}|^{1/n} = \frac{1}{e\gamma\log\alpha^{-1}} .
\label{f2tratio}
\end{equation}
Therefore expansion \eqref{large-j_Fj} of Prop. \ref{ProductFormula} for our particular application is not convergent 
and is in fact asymptotic in the sense of large $\gamma$ and small $\alpha$.
\end{proposition}
\begin{proof}
Our strategy is to adapt the method of \cite{Fel_1932} to our situation.
Starting with \eqref{fCoeff<-dCoeff} rephrased slightly differently we deduce the inequality
\begin{equation}
	n|f_{n}| \leq |d_{n}| + \sum_{{p|n \atop 1\leq p < n}}p|f_{p}|^{n/p} ,
\label{divisorINEQ}
\end{equation}
summing over the proper divisors $p$ of $n$.
Let $ p_{1} $ denote the divisor contributing the greatest term to the above sum. 
Then if we use $n/2$ as an upper bound for the number of proper divisors of $n$ we have
\begin{equation}
	\frac{n|f_{n}|}{|d_{n}|} 
	\leq 1 + \tfrac{1}{2}n p_{1} \frac{|f_{p_{1}}|^{n/p_{1}}}{|d_{n}|} 
	= 1 + \tfrac{1}{2}n p_{1}^{1-n/p_{1}}\frac{|d_{p_{1}}|^{n/p_{1}}}{|d_{n}|} \left( \frac{p_{1}|f_{p_{1}}|}{|d_{p_{1}}|} \right)^{n/p_{1}} .
\label{raw_inequality}
\end{equation}
From Prop.\ref{dGrowth} we know
\begin{equation*}
	\frac{|d_{p_{1}}|^{n/p_{1}}}{|d_{n}|} \leq \frac{p_{1}^{n/p_{1}}}{n}\frac{|t_{p_{1}}|^{n/p_{1}}}{|t_{n}|}\left( 1+{\rm O}(n^{-1},p_{1}^{-1}) \right) ,
\end{equation*}
and from Prop.\ref{tLUBounds},
\begin{equation*}
	\frac{|t_{p_{1}}|^{n/p_{1}}}{|t_{n}|} \leq 
	\left[  \sqrt{2\pi}\frac{e^{-\gamma}}{\Gamma(\gamma)}\left(\frac{1-\alpha}{\log\alpha^{-1}}\right)^{\gamma} \right]^{n/p_{1}-1}
	\left[ \frac{(p_{1}+\gamma)^{n/p_{1}}}{n+\gamma} \right]^{\gamma-1/2} 
	\left(\frac{p_{1}+\gamma}{n+\gamma}\right)^{n} .
\end{equation*}
If we make the simplifying definitions
\begin{equation*}
	C = \frac{\sqrt{2\pi}}{\Gamma(\gamma)}\left(\frac{1-\alpha}{\log\alpha^{-1}}\right)^{\gamma} ,
	\qquad
	V_{p_{1}} := C^{1-p_{1}/n} \left[ \frac{p_{1}+\gamma}{(n+\gamma)^{p_{1}/n}}\right]^{\gamma-1/2} \frac{p_{1}|f_{p_{1}|}}{|d_{p_{1}}|} ,
\end{equation*}
we deduce that inequality \eqref{raw_inequality} becomes
\begin{equation*}
	V_{n} \leq 1 + \tfrac{1}{2}p_{1}\left(\frac{p_{1}+\gamma}{n+\gamma}\right)^{n} V_{p_{1}}^{n/p_{1}} .
\end{equation*}
We will now apply this inequality recursively by constructing a sequence of proper divisors $ p_{0}:=n, p_{1}, p_{2}, \dots ,p_{k}, p_{k+1}, \ldots $
where $ p_{k+1} | p_{k} $, $ p_{k+1} < p_{k} $ and $ p_{k+1} $ is the proper divisor with the largest contribution $ p_{k+1}|f_{p_{k+1}}|^{p_{k}/p_{k+1}} $
to the second term on the right-hand side of \eqref{raw_inequality} out of all such proper divisors. 
This sequence will terminate because of the strict inequality applying and all satisfy $ p_{k} \geq 1 $, and we denote the terminal divisor $p_{m+1}=1 $.
However the utility of this construction is that we do not need to identify these divisors because of the simple inequality $ p_{k+1} \leq p_{k}/2 $.
Therefore starting from the top and descending we can generate a sequence of upper bounds $ p_{k} \leq 2^{-k}n $ 
whereas if we start from the terminating divisor and ascend we generate a sequence of lower bounds $ 2^{m+1-k} \leq p_{k} $.
We note that combining these bounds we deduce that $ 2^{m+1} \leq n $ and we find the upper bound $ m+1 \leq \left\lceil \displaystyle\frac{\log n}{\log 2} \right\rceil $.
To further simplify matters we note that
\begin{equation*}
	\left( \frac{p_{k+1}+\gamma}{p_{k}+\gamma} \right)^{p_{k}} \leq e^{\gamma}\:2^{-p_{k}} .
\end{equation*}
To see how this inequality arises take the even case $p_{k}=2l$ first:
\begin{align*}
	p_{k+1}+\gamma	& \leq \tfrac{1}{2}p_{k}+\gamma = \tfrac{1}{2}(p_{k}+2\gamma) ,
\\
	\implies \frac{p_{k+1}+\gamma}{p_{k}+\gamma} & \leq \tfrac{1}{2}\left[ 1+\gamma(p_{k}+\gamma)^{-1} \right]  \leq \tfrac{1}{2}\left[ 1+\gamma p_{k}^{-1} \right] ,
\\
	\implies \left( \frac{p_{k+1}+\gamma}{p_{k}+\gamma} \right)^{p_{k}}	& \leq 2^{-p_{k}}\left[ 1+\gamma p_{k}^{-1} \right]^{p_{k}} \leq 2^{-p_{k}}e^{\gamma} ,
\end{align*}
whereas for the odd case $p_{k}=2l+1$: $p_{k+1} \leq l+1/2$ actually means $p_{k+1} \leq l $ thus 
\begin{align*}
	p_{k+1}+\gamma	& \leq \tfrac{1}{2}p_{k}-\tfrac{1}{2}+\gamma = \tfrac{1}{2}\left( p_{k}+\gamma \right) + \tfrac{1}{2}(\gamma-1),
\\
	\implies \frac{p_{k+1}+\gamma}{p_{k}+\gamma} & \leq \tfrac{1}{2}\left[ 1+(\gamma-1)(p_{k}+\gamma)^{-1} \right]  \leq \tfrac{1}{2}\left[ 1+\gamma p_{k}^{-1} \right] ,
\end{align*}
and the remainder of the steps follows the even case. Thus our final inequality becomes a nested recursion
\begin{equation*}
	V_{p_{k}} \leq 1 + \tfrac{1}{2} p_{k+1} 2^{-p_{k}}V_{p_{k+1}}^2, \quad k=0, 1, \ldots m+1 .
\end{equation*}
Iterating this we find $ V_{n} \leq 1+2^{-g_{0}}x_{m+1} $ where the exponents are $ g_{k}=k+1-m+2^{m+1-k} $ 
and the numerators are generated by the recurrence $ x_{n}=1+x_{n-1}^2 $ with $x_1=1$.
This is the OEIS sequence \cite{OEIS} {\tt A003095}, where $x_{n}$ counts the number of binary trees of height less than or equal to $n$ 
and has the asymptotic growth as $ n\to\infty $ of $ x_{n} \mathop{\sim}\limits_{n \to \infty} c^{2n} $ with $ c=1.2259024435287485386279474959130085213... $.
In conclusion we find exponential convergence of $ V_{n} $ to unity
\begin{equation*}
	V_{n} \leq 1 + 2^{-2^{m+1}+m-1}c^{2m+2} = 1 + \tfrac{1}{2}\left(\tfrac{1}{2}c^2\right)^{n}\left( \tfrac{1}{2}n \right)^{\log 2},
\end{equation*}
and \eqref{f2dratio} follows.
\end{proof}

\section{Gamma Function Products}\label{ProductsofGamma}
\setcounter{equation}{0}

In our final section we employ the product expansion \eqref{large-j_Fj} as an asymptotic expansion in order to derive some leading order approximations for the $\alpha$-Sun density.
At this juncture it is important to clarify a distinction between the true and rigorous convergent nature of the $j$-product in \eqref{MBdensity} and what we are proposing here:
each factor of $F_{j}$ is replaced by an internal product of $m$ factors via \eqref{large-j_Fj} up to a fixed cut-off, 
and the order of products reversed in order to perform the $j$ product exactly, 
so that one has a tractable result for the Mellin transform $ H(s) $ and can compute the inverse Mellin transform explicitly.
We will only carry this out for the first two orders and compare this with the recent related result of Simon \cite{Sim_2023}.

To begin with we recall some classical results concerning evaluations of infinite products of degree-$m$ monomial factors which evaluate as $m$-fold Gamma function products.
\begin{lemma}
Let the $m$-th root of unity be denoted by $ \omega_{m} = e^{2\pi i/m} $. Then for 
$ z \notin \N\; \omega^{-j}_{m} $ with $ j=0,\ldots,m-1 $, $ m\in\N $, $ m\geq 2 $ we have \cite{Hansen_1975}
\begin{equation}
	\prod_{n=1}^{\infty} \frac{n^{m}}{n^{m}-z^{m}} = \prod_{j=0}^{m-1} \Gamma(1-\omega^{j}_{m}z) .
\label{GammaProducts}
\end{equation}
Furthermore let $ |z|< 1 $ \cite{Wan_2020} then the above product can also be evaluated as
\begin{equation}
	\prod_{j=0}^{m-1} \Gamma(1-\omega^{j}_{m}z) = \exp\left( \sum_{k=1}^{\infty}\frac{\zeta(mk)}{k}z^{mk} \right) .
\end{equation}
\end{lemma}

Our first technical result is the evaluation of infinite products resulting from the truncation of the $m$-expansion in \eqref{large-j_Fj}.
\begin{corollary}
Let $ \omega^{(k+1/2)}_{m} = e^{2\pi i/m\cdot(k+\frac{1}{2})} $ for $m\in \N$, $k\in \Z_{\geq 0}$. 
Furthermore let $t\in \C$ excluding all $1-f_{m}^{1/m}\omega^{(k+1/2)}_{m}$. Then
\begin{multline}
	\prod_{j=1}^{\infty} \frac{1+f_{1} j^{-1}}{1+f_{1} (j-t)^{-1}} \ldots \frac{1+f_{m} j^{-m}}{1+f_{m} (j-t)^{-m}}
\\
	= \frac{\Gamma(1-t+f_{1})}{\Gamma(1-t)\Gamma(1+f_{1})}
	  \cdots
	  \prod_{k=0}^{m-1}\frac{\Gamma(1-t-f_{m}^{1/m}\omega^{(k+1/2)}_{m})}{\Gamma(1-t)\Gamma(1-f_{m}^{1/m}\omega^{(k+1/2)}_{m})}	.
\label{GammaProd}
\end{multline}
\end{corollary}
\begin{proof}
We compute the following finite product $J\in \N$ for $m>1$ using \eqref{GammaProducts},
\begin{align*}
	\prod_{j=1}^{J} \left( 1+\frac{f_{m}}{(j-t)^{m}} \right)	
	& = \prod_{k=0}^{m-1}\frac{\Gamma(J+1-t-f_{m}^{1/m}\omega^{(k+1/2)}_{m})}{\Gamma(J+1-t)}
			\frac{\Gamma(1-t)}{\Gamma(1-t-f_{m}^{1/m}\omega^{(k+1/2)}_{m})} ,
\\	
	& \mathop{\to}\limits_{J \to \infty} 
			(J+1)^{-f_{m}^{1/m}\sum_{k=0}^{m-1}\omega^{(k+1/2)}_{m}} \prod_{k=0}^{m-1}\frac{\Gamma(1-t)}{\Gamma(1-t-f_{m}^{1/m}\omega^{(k+1/2)}_{m})} ,
\\
	& = \prod_{k=0}^{m-1}\frac{\Gamma(1-t)}{\Gamma(1-t-f_{m}^{1/m}\omega^{(k+1/2)}_{m})} .
\end{align*}
The first equality above is derived in the same way as the proof of \eqref{GammaProducts}, that is by performing a root-of-unity factorisation of each factor
which results in a $m$-fold product of polynomials all of whom are linear in $j$, and for each of those linear factors the product over $j$ can be evaluated as a ratio of Gamma functions.
The final equality holds because of the sum identity $ \sum_{k=0}^{m-1}\omega^{(k+1/2)}_{m}=0 $.
For $ m=1 $ one can use the fact that only ratios are actually required in $ \prod_{j=1}^{J} \left( 1+\frac{f_{1}}{j} \right)/ \left( 1+\frac{f_{1}}{(j-t)} \right) $ 
and the convergence as $ J\to \infty $ follows as this ratio is $ 1+{\rm O}(j^{-2}) $.  
\end{proof}

Let us first write products and ratios of Gamma functions in the following concise notation: 
for arbitrary complex indeterminates $ a_1, a_2, \ldots $, $ b_1, b_2, \ldots $ subject to $ a_k \neq -\Z_{\geq 0} $,
\begin{equation*}
	\Gamma\left(
		\begin{array}{ccc}
				a_1 & a_2 & \ldots \\
				b_1 & b_2 & \ldots 
		\end{array}\right)
	= \frac{\Gamma(a_1)\Gamma(a_2)\cdots}{\Gamma(b_1)\Gamma(b_2)\cdots} .
\end{equation*}
A similar array notation has been employed for the parameters of the $ {}_PF_{Q} $ hypergeometric functions.
\begin{proposition}\label{density_approx}
Let $ 0\leq \alpha < 1$.
At the first order $m=1$ we compute the Mellin inversion and find a previously given result (see Eq. 3.61 in \cite{WG_2023}),
\begin{equation}
	h(x) \mathop{\sim}\limits_{\gamma \gg 1 \gg \alpha} \frac{\gamma(1-\alpha)^{-\gamma(1-\alpha)^{-1}}}{\Gamma(\frac{1}{1-\alpha})} x^{-1-\gamma/(1-\alpha)} e^{-(1-\alpha)^{-\gamma}x^{-\gamma}} .
\label{1stApproxDensity}
\end{equation}
At the second order we have
\begin{multline}
	h(x) \mathop{\sim}\limits_{\gamma \gg 1 \gg \alpha} \frac{\gamma x^{-1}}{\Gamma\left( \frac{1}{1-\alpha},1+\frac{\alpha^{1/2}\gamma^{-1/2}}{1-\alpha},1-\frac{\alpha^{1/2}\gamma^{-1/2}}{1-\alpha} \right)}
	\Bigg\{
\\
	\Gamma\left(
	\begin{array}{cc}
			\frac{\left[\alpha^{1/2}\gamma^{-1/2}-\alpha\right]}{1-\alpha} & -\frac{\left[\alpha^{1/2}\gamma^{-1/2}+\alpha\right]}{1-\alpha} \\
			-\frac{\alpha}{1-\alpha} & -\frac{\alpha}{1-\alpha}
	\end{array}\right)
	\left[(1-\alpha)x\right]^{-\gamma/(1-\alpha)}
\\ \times
	{}_2F_{2}\left(
		\begin{array}{cc}
				\frac{1}{1-\alpha}, & \frac{1}{1-\alpha} \\
				\frac{\left[1-\alpha^{1/2}\gamma^{-1/2}\right]}{1-\alpha}, & \frac{\left[1+\alpha^{1/2}\gamma^{-1/2}\right]}{1-\alpha}
		\end{array};-(1-\alpha)^{-\gamma}x^{-\gamma} \right)
\\
	+\Gamma\left(
	\begin{array}{cc}
			\frac{\left[-\alpha^{1/2}\gamma^{-1/2}+\alpha\right]}{1-\alpha} & -2\frac{\alpha^{1/2}\gamma^{-1/2}}{1-\alpha} \\
			-\frac{\alpha^{1/2}\gamma^{-1/2}}{1-\alpha} & -\frac{\alpha^{1/2}\gamma^{-1/2}}{1-\alpha}
	\end{array}\right)
	\left[(1-\alpha)x\right]^{-\gamma\left[1+(1-\alpha)^{-1}\alpha^{1/2}\gamma^{-1/2}\right]}
\\ \times
	{}_2F_{2}\left(
		\begin{array}{cc}
				1+\frac{\alpha^{1/2}\gamma^{-1/2}}{1-\alpha}, & 1+\frac{\alpha^{1/2}\gamma^{-1/2}}{1-\alpha} \\
				1-\frac{\left[\alpha-\alpha^{1/2}\gamma^{-1/2}\right]}{1-\alpha}, & 1+2\frac{\alpha^{1/2}\gamma^{-1/2}}{1-\alpha}
		\end{array};-(1-\alpha)^{-\gamma}x^{-\gamma} \right)
\\
	+\Gamma\left(
	\begin{array}{cc}
			\frac{\left[\alpha^{1/2}\gamma^{-1/2}+\alpha\right]}{1-\alpha} & 2\frac{\alpha^{1/2}\gamma^{-1/2}}{1-\alpha} \\
			\frac{\alpha^{1/2}\gamma^{-1/2}}{1-\alpha} & \frac{\alpha^{1/2}\gamma^{-1/2}}{1-\alpha}
	\end{array}\right)
	\left[(1-\alpha)x\right]^{-\gamma\left[1-(1-\alpha)^{-1}\alpha^{1/2}\gamma^{-1/2}\right]}
\\ \times
	{}_2F_{2}\left(
		\begin{array}{cc}
				1-\frac{\alpha^{1/2}\gamma^{-1/2}}{1-\alpha}, & 1-\frac{\alpha^{1/2}\gamma^{-1/2}}{1-\alpha} \\
				1-\frac{\left[\alpha+\alpha^{1/2}\gamma^{-1/2}\right]}{1-\alpha}, & 1-2\frac{\alpha^{1/2}\gamma^{-1/2}}{1-\alpha}
		\end{array};-(1-\alpha)^{-\gamma}x^{-\gamma} \right)
	\Bigg\} .
\label{2ndApproxDensity}
\end{multline}
\end{proposition}
\begin{proof}
We give details just for the first order case as this is typical for the general situation. Applying \eqref{GammaProd} to the specific case at hand we find that
\begin{equation*}
	\prod_{j=1}^{\infty}\frac{1+f_1 j^{-1}}{1+f_1 (j-t)^{-1}} = \frac{\Gamma(\frac{1}{1-\alpha}-t)}{\Gamma(1-t)\Gamma(\frac{1}{1-\alpha})} .
\end{equation*}
Replacing the factor $ \prod_{j=1}^{\infty} F_{j-t}/F_{j} $ in \eqref{MBdensity} with the above expression we have the Mellin-Barnes integral
\begin{equation*}
	h(x) \mathop{\sim}\limits_{\gamma \gg 1 \gg \alpha} \frac{\gamma}{2\pi ix}\int_{c-i\infty}^{c+i\infty} dt\; x^{-\gamma t}(1-\alpha)^{-\gamma t}\frac{\Gamma(\frac{1}{1-\alpha}-t)}{\Gamma(\frac{1}{1-\alpha})} ,
\end{equation*}
with $c<1$. We note that the only singularities of the integrand are simple poles at $ t_{l}=l+(1-\alpha)^{-1} $, $ l\geq 0 $.
However because the integrand has an exponential decay as $ t\to c\pm i\infty $ arising from the Gamma function factor $ \Gamma((1-\alpha)^{-1}-t) $
one can fold the upper and lower arms of the contour over the positive-$t$ axis and thus enclosing all of these poles. 
One computes the residues in the standard way \cite{Davies_2002} as
\begin{equation*}
	h(x) \mathop{\sim}\limits_{\gamma \gg 1 \gg \alpha} \frac{\gamma x^{-1}}{\Gamma(\frac{1}{1-\alpha})} \sum_{l=0}^{\infty}\frac{(-1)^l}{l!} \left[(1-\alpha)x\right]^{-\gamma(l+(1-\alpha)^{-1})} ,
\end{equation*}
and the summation is trivial, leading to \eqref{1stApproxDensity}.
To include the second order factors as well one has two new sequences of simple poles at $ t_{l}=l+1\pm (1-\alpha)^{-1}\alpha^{1/2}\gamma^{-1/2} $ in addition to the first order set.
We assume that these three sequences do not overlap, as in the case of generic $\alpha, \gamma$.
After some computation we recognise the series definition of the $ {}_2F_2 $ hypergeometric function, and simplifying we arrive at \eqref{2ndApproxDensity}.
\end{proof}

\begin{corollary}
The small $x\to 0^{+}$ asymptotic form of the second order approximation to the density \eqref{2ndApproxDensity} assumes the simple form 
\begin{equation}
	h(x) \mathop{\sim}\limits_{x \to 0^+} \frac{\gamma(1-\alpha)^{-\gamma/(1-\alpha)}}{\Gamma\left(\frac{1}{1-\alpha}, 1+\frac{\alpha^{1/2}\gamma^{-1/2}}{1-\alpha}, 1-\frac{\alpha^{1/2}\gamma^{-1/2}}{1-\alpha}\right)}
	x^{-1-\gamma/(1-\alpha)} e^{-(1-\alpha)^{-\gamma}x^{-\gamma}} ,
\label{2ndAsympDensity}
\end{equation}
which has precisely the same $x$-dependence as \eqref{1stApproxDensity} and only differs in the constant factor.
\end{corollary}
\begin{proof}
Utilising the large argument asymptotic leading order terms for the hypergeometric function, \cite[Eq. 16.11.7]{DLMF} with $\kappa=1$, $\nu=a_1+a_2-b_1-b_2$, 
in \eqref{2ndApproxDensity}
\begin{multline*}
	\frac{\Gamma(a_1)\Gamma(a_2)}{\Gamma(b_1)\Gamma(b_2)} {}_2F_2\left(
			\begin{array}{cc} 
					a_1, & a_2 \\
					b_1, & b_2
			\end{array}; z \right) \mathop{\sim}\limits_{z \to \infty} z^{\nu}e^{z}
\\
			 + \frac{\Gamma(a_1)\Gamma(a_1-a_2)}{\Gamma(b_1-a_1)\Gamma(b_2-a_1)}\left( e^{\pm \pi i}z \right)^{-a_1}
			 + \frac{\Gamma(a_2)\Gamma(a_2-a_1)}{\Gamma(b_1-a_2)\Gamma(b_2-a_2)}\left( e^{\pm \pi i}z \right)^{-a_2} ,
\end{multline*}
we find significant simplification occurs because all three terms end up with the same $x$-dependence and the amalgamation of their coefficients into a single common factor, resulting in \eqref{2ndAsympDensity}.
\end{proof}

In concluding we present clear numerical evidence of our approximations versus the recent result in \cite[Thm. 2]{Sim_2023}, 
which gives the asymptotic form of the density as $x \to 0^+$ in \eqref{h_x2zero} and \eqref{constant}.
Note that \eqref{1stApproxDensity} is a global approximation to the density as the constant factor is its true normalisation 
even though the $x$-dependent part is identical to the true asymptotic form as $ x\to 0^{+} $, as given by the result of Theorem 2 in \cite{Sim_2023}.
Such approximations are then a large $\gamma$ and small $\alpha$ asymptotic development, in the sense of $(1-\alpha)^{-1}$ being of ${\rm O}(1)$.
When such conditions are violated, as is apparent in the plots Fig. \ref{fig:1}-\ref{fig:4}, then the approximation can vanish and change sign. 
Another point to note from Fig. \ref{fig:5} is that the constant $ c(\alpha,\gamma) $ \eqref{constant} diverges to $ 0 $ or $ \infty $ as $ \alpha \to 1^{-} $ when $ \gamma < 1$ or $ \gamma\geq 1$ respectively.  

\begin{figure}[H]
\centering
\begin{minipage}{.5\textwidth}
  \centering
  \includegraphics[width=.9\linewidth]{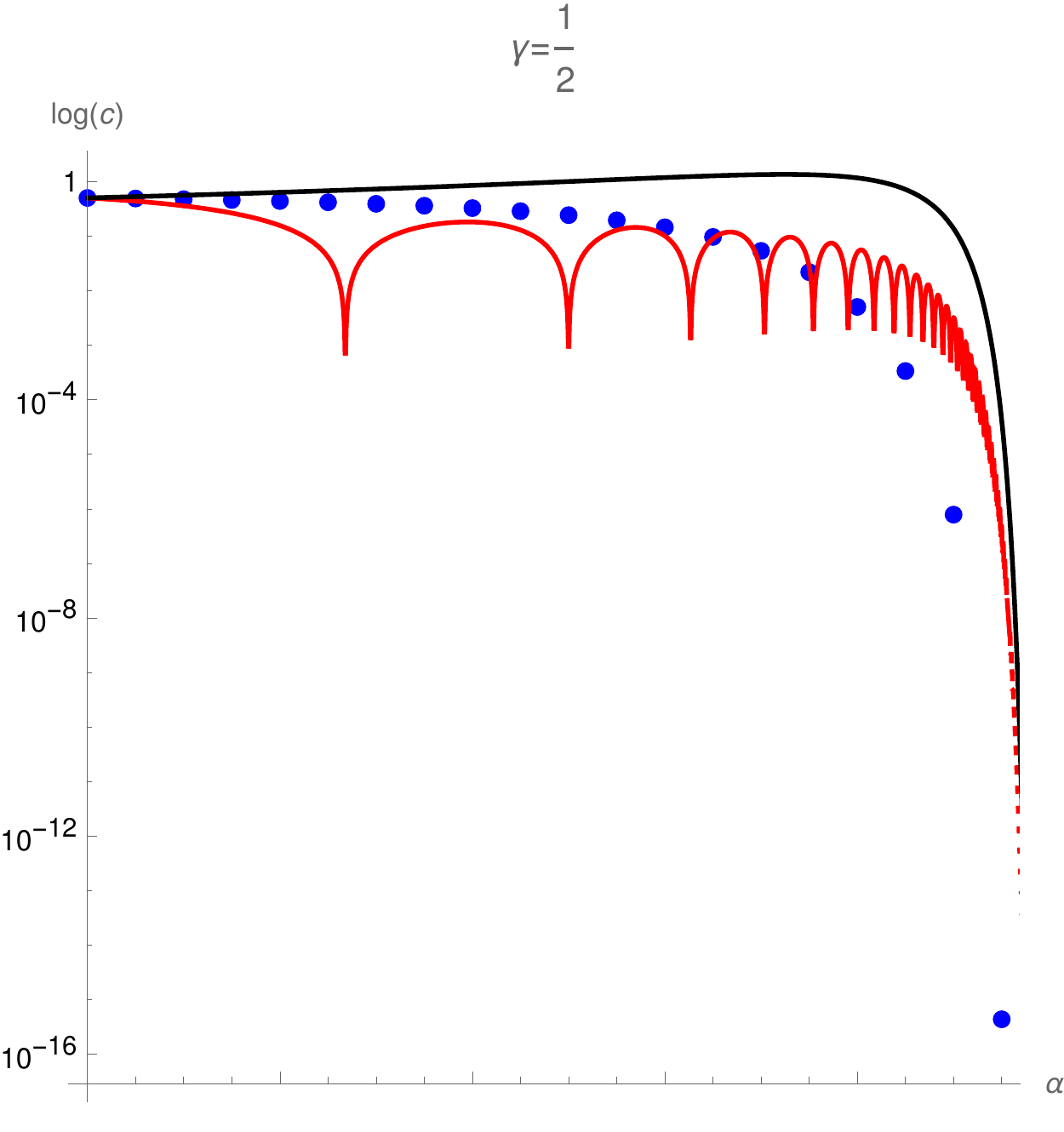}
  \captionof{figure}{Plot of the constant factor in the $x\to 0+$ asymptotic form of $ h(x) $ versus $ \alpha $ for $\gamma=\tfrac{1}{2}$. 
  Blue dots are the exact $c(\alpha,\gamma)$ \eqref{constant}, 
  the black line is the first order approximation \eqref{1stApproxDensity} 
  and the red line is the absolute value of the second order approximation \eqref{2ndAsympDensity}.}
  \label{fig:1}
\end{minipage}%
\begin{minipage}{.5\textwidth}
  \centering
  \includegraphics[width=.9\linewidth]{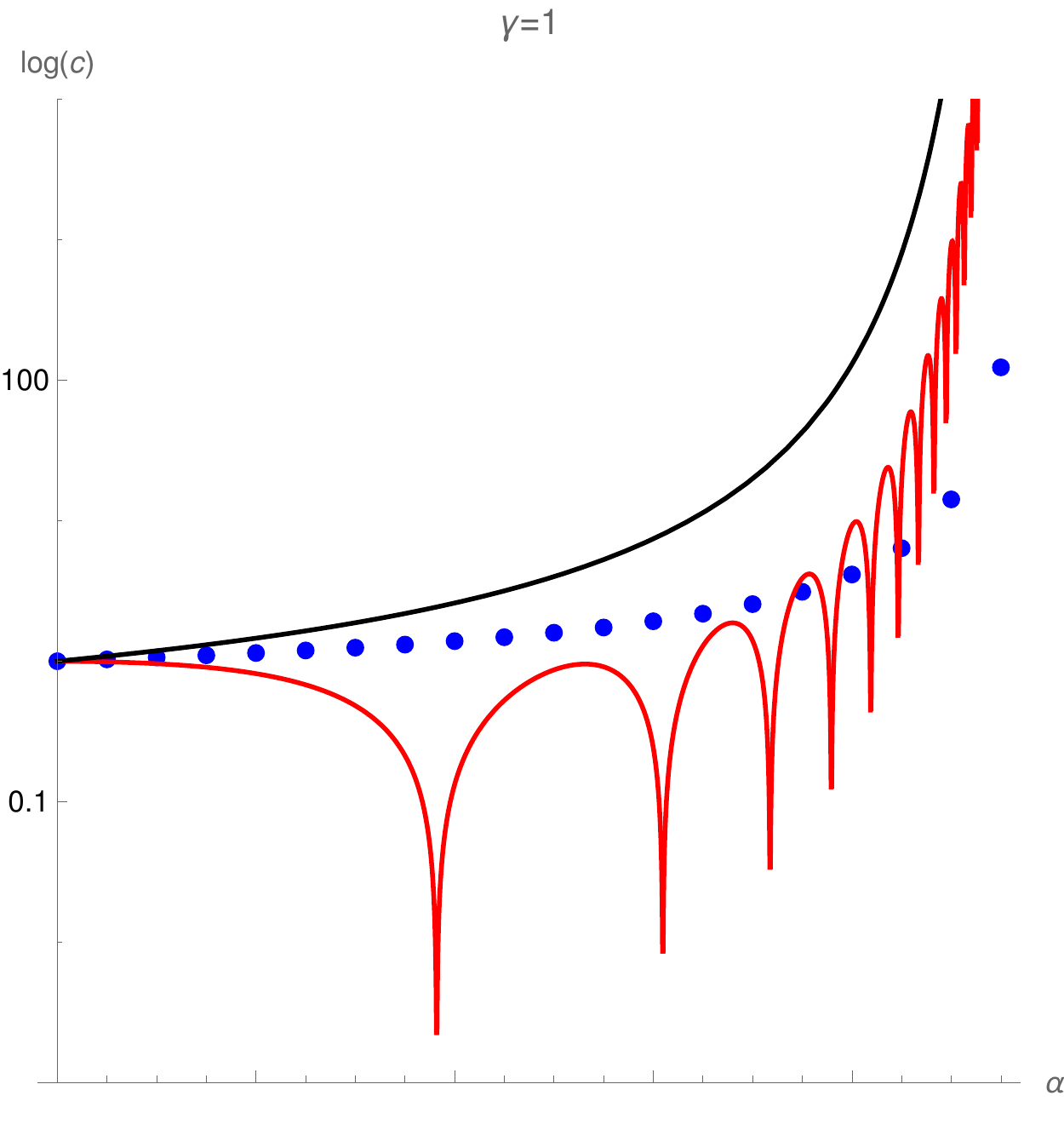}
  \captionof{figure}{As per Fig. \ref{fig:1} for $\gamma=1$.}
  \label{fig:2}
\end{minipage}
\end{figure}

\begin{figure}[H]
\centering
\begin{minipage}{.5\textwidth}
  \centering
  \includegraphics[width=.9\linewidth]{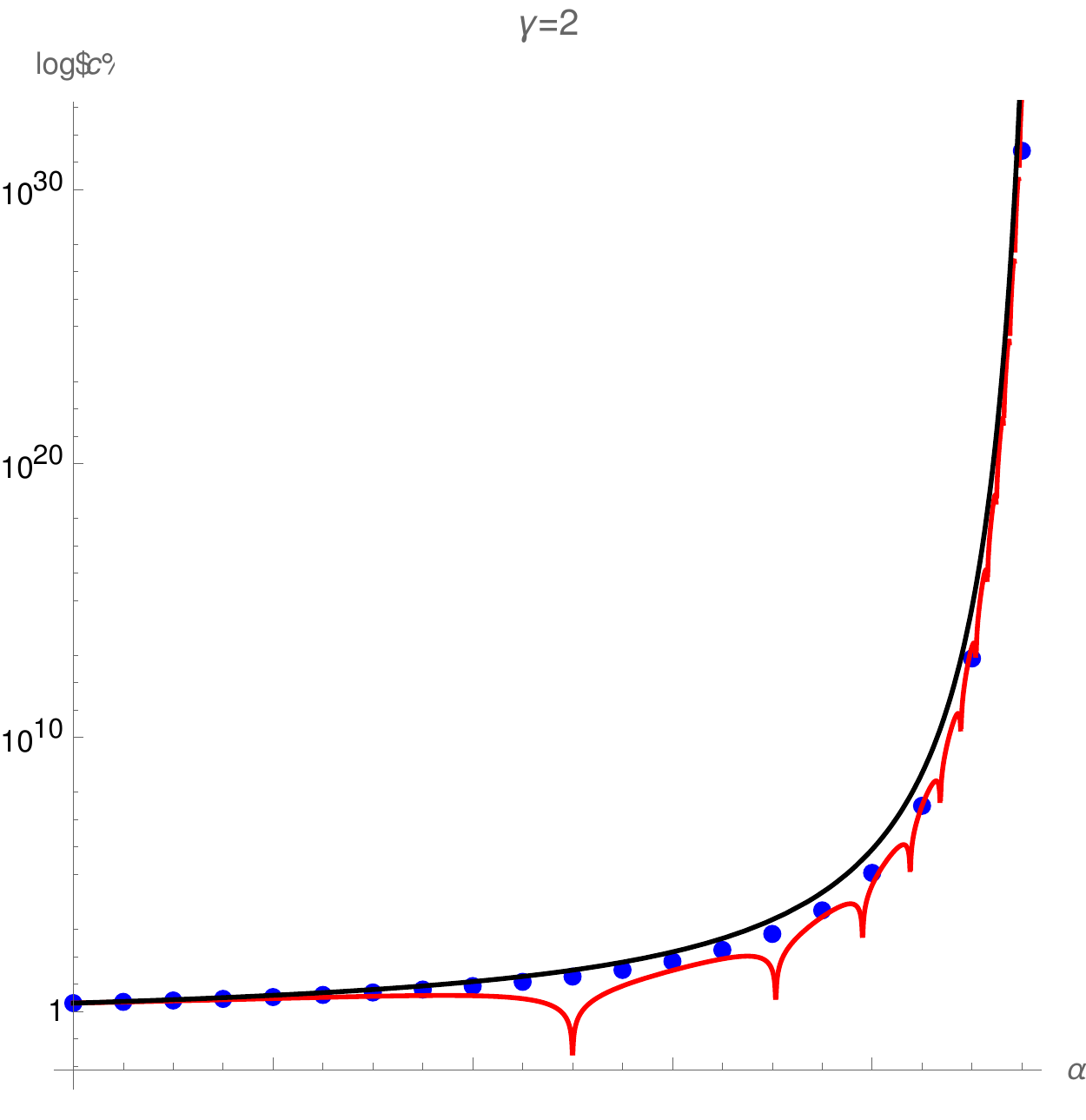}
  \captionof{figure}{As per Fig. \ref{fig:1} for $\gamma=2$.}
  \label{fig:3}
\end{minipage}%
\begin{minipage}{.5\textwidth}
  \centering
  \includegraphics[width=.9\linewidth]{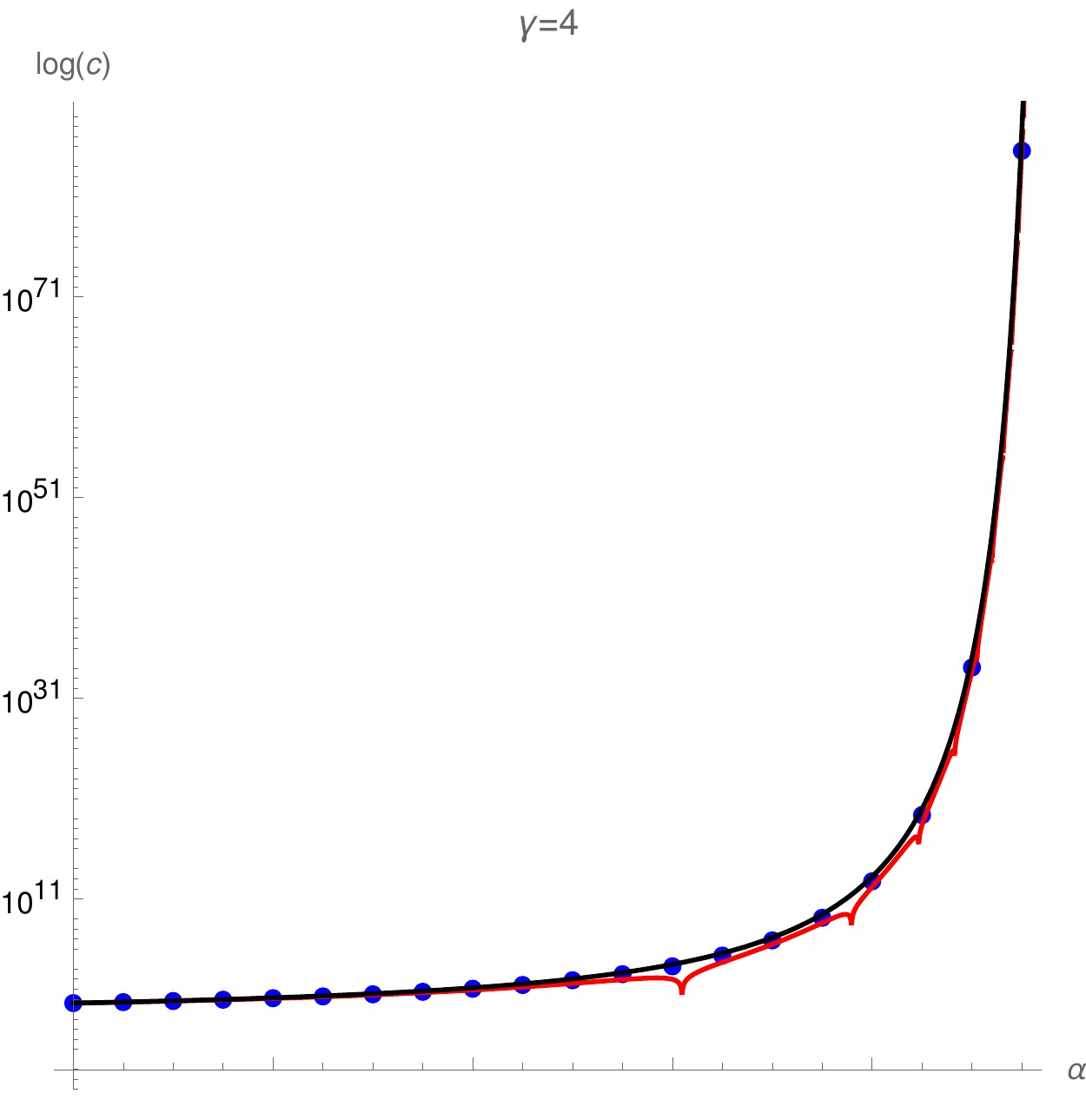}
  \captionof{figure}{As per Fig. \ref{fig:1} for $\gamma=4$.}
  \label{fig:4}
\end{minipage}
\end{figure}

\begin{figure}[H]
\centering
  \includegraphics[width=100mm,height=140mm]{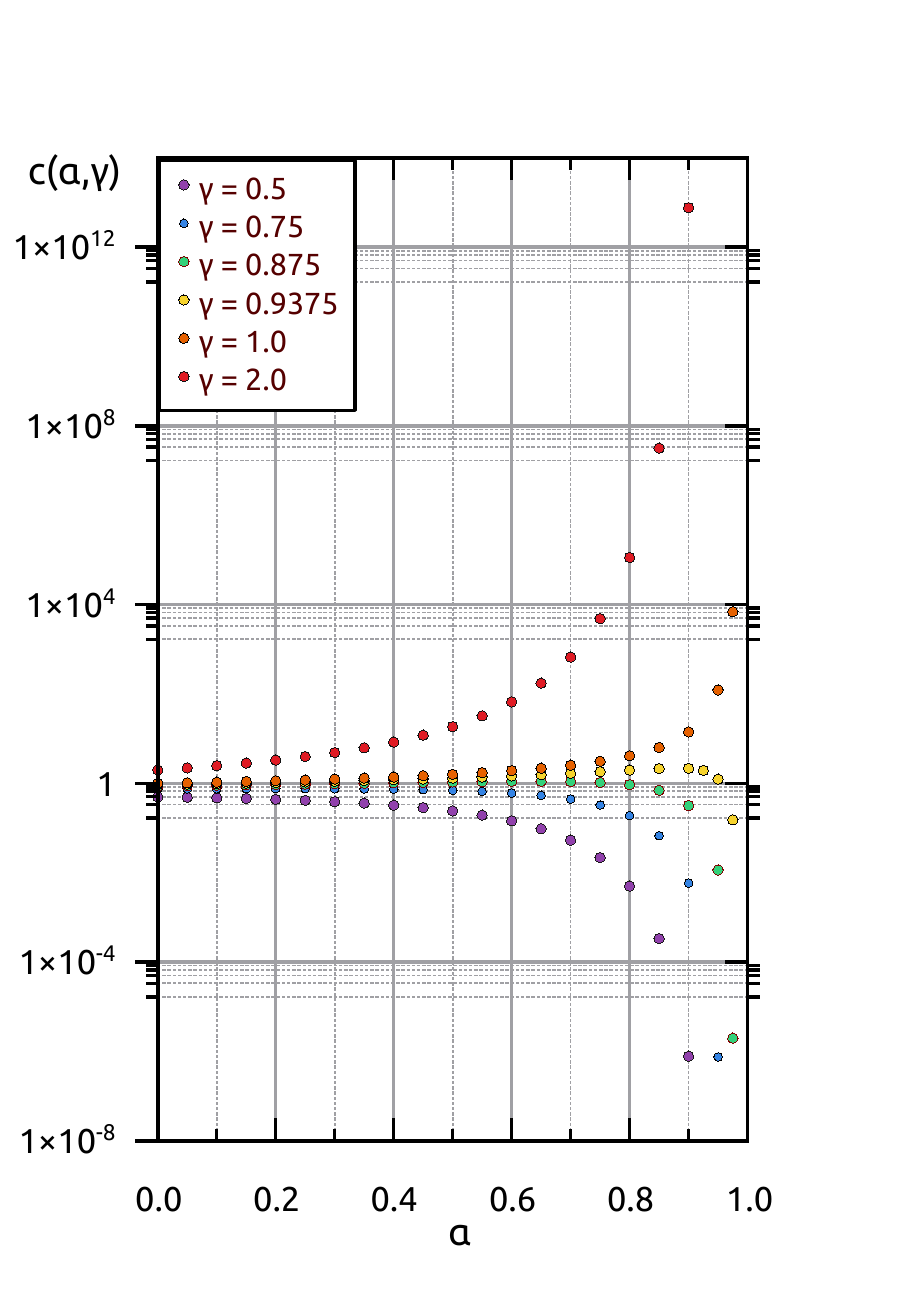}
\caption{Plot of the constant factor \eqref{constant} in the $x\to 0+$ asymptotic form of $ h(x) $ versus $ \alpha\in [0,1) $ for $\gamma=\tfrac{1}{2},\ldots,2$.}
\label{fig:5}       
\end{figure}

\section{Acknowledgements}
The author thanks Howard Cohl for the invitation to present a preliminary version of the current work to the 2022 AMS fall sectional meeting {\it hypergeometric functions, q-series and generalizations} in Salt Lake City, Utah, and thanks Cindy Greenwood and Thomas Simon for much appreciated correspondence.
He would also like express gratitude to the School of Mathematics and Statistics, University of Melbourne, for hosting his visit during the final stages of this study.
Lastly the presentation and clarity of this manuscript has greatly benefited from the attention to detail and thoroughness of the referees.  

\bibliographystyle{plain}

\def\cprime{$'$} \def\cprime{$'$} \def\cprime{$'$}

\end{document}